\documentclass[a4paper,11pt]{article}
\usepackage{graphicx}
\usepackage{amsfonts,amsmath,amssymb,amsthm,amscd}
\usepackage{amssymb}
\usepackage{fancyhdr}
\usepackage{mathtools}
\usepackage{indentfirst}
\usepackage{algorithm}
\usepackage{algpseudocode}
\usepackage{placeins}
\usepackage{listings}
\usepackage{lipsum}
\usepackage{caption}
\usepackage{subfigure}
\usepackage{xcolor}
\usepackage{comment}
\usepackage{url}
\usepackage{soul}
\usepackage{bm}
\usepackage{hyperref}
\usepackage{multirow}
\usepackage{multicol}
\usepackage{siunitx}

\DeclareMathOperator*{\argmin}{arg\,min}
\DeclareMathOperator{\Var}{var_\mu}
\DeclareMathOperator{\Cov}{cov_\mu}
\DeclareMathOperator{\RQ}{\mathcal{Q}}
\DeclareMathOperator{\RmapNN}{\mathcal{R}}

\newcommand{\ESuperscript}{{\mathrm{data}, \mathcal{I}}}
\newcommand{\ESuperscriptJ}{{\mathrm{data}, \mathcal{I}^{(\mathnormal{j})}}}

\DeclareMathOperator{\EVar}{\Var^\ESuperscript}
\DeclareMathOperator{\ECov}{{\Cov^{\ESuperscript}}}
\DeclareMathOperator{\EE}{\mathbf{E}^{\ESuperscript}}
\DeclareMathOperator{\ERQ}{\RQ^{\ESuperscript}} 

\DeclareMathOperator{\EEJ}{\mathbf{E}^{\ESuperscriptJ}}
\DeclareMathOperator{\ERQJ}{\RQ^{\ESuperscriptJ}}

\newtheorem{remark}{Remark}
\newtheorem{prop}{Proposition}
\newtheorem{assump}{Assumption}
\newtheorem{lemma}{Lemma}

\newtheorem{thm}{Theorem}

\begin{document}
  \title{Solving eigenvalue PDEs of metastable diffusion processes using artificial neural networks}
  \date{}
  \author{
     Wei Zhang~\thanks{Zuse Institute Berlin, Takustrasse 7, 14195 Berlin, Germany.\  Email: wei.zhang@fu-berlin.de}
\and
    Tiejun Li~\thanks{Laboratory of Mathematics and Applied Mathematics (LMAM)
    and School of Mathematical Sciences, Peking University, Beijing 100871. P.R. China.\ Email: tieli@pku.edu.cn} 
    \and
  Christof Sch\"utte~\thanks{Institut f\"ur Mathematik, Freie Universit\"at Berlin and Zuse Institute Berlin, D-14195 Berlin, Germany; christof.schuette@fu-berlin.de} 
  }
  \maketitle
  \begin{abstract}
    In this paper, we consider the eigenvalue PDE problem of the infinitesimal generators of metastable diffusion processes. We propose a numerical
    algorithm based on training artificial neural networks for solving the leading eigenvalues and eigenfunctions of such high-dimensional eigenvalue problem. 
The algorithm is able to find multiple leading eigenpairs by solving a single training task.
  It is useful in understanding the dynamical behaviors of metastable processes on large timescales.  We demonstrate the
    capability of our algorithm on a high-dimensional model problem, and on the simple molecular system alanine dipeptide.
  \end{abstract}

  \begin{keywords}
   eigenvalue PDE, metastable process, molecular dynamics, artificial neural network, variational characterization
  \end{keywords}
 
% \begin{AMS}
%   35P99, 35Q92, 68T07
% \end{AMS}

  \section{Introduction}
  \label{sec-intro}
Understanding the dynamics of molecular systems is often a challenging task
  due to the high dimensionality of the systems and their extraordinarily
  complex dynamical behavior. In the last decades, considerable amount of efforts have been devoted to
  developing high-performance numerical packages and new simulation
  techniques, leading to the rapid advance of the capability of molecular
  dynamics simulations in generating trajectory data. At the same time, many data-based numerical
  approaches have emerged, which allow to efficiently study the molecular
  kinetics through analyzing the trajectory data obtained from molecular dynamics simulations. 
  A large class of these approaches for trajectory data analysis are based on the theory of the transfer
  operator~\cite{SchuetteHuisingaDeuflhard2001} or Koopman operator~\cite{Koopmanism}, hence termed operator-based approaches, in which
  the molecular system is analyzed by estimating the dominant eigenvalues and
  their corresponding eigenfunctions of the operator that is discretized using trajectory data. Notable examples are the variational approach to conformational
  dynamics~\cite{frank_feliks_mms2013,feliks_variational_jctc_2014} and its
  linear version, time lagged independent component analysis
  (tICA)~\cite{tica}, the variational approach for Markov processes
  (VAMP)~\cite{vamp}, the extended dynamic mode decompositions~\cite{edmd,klus-koopman,EDMD-Klus}, Markov state models (MSMs)~\cite{msm_generation,CHODERA-noe-msm-2014,husic-pande-msm}, etc. Recent development in these directions includes the kernel-tICA~\cite{kernel-tica} using kernel method, the deep learning frameworks VAMPNets~\cite{vampnet} and the state-free reversible VAMPNets (SRVs)~\cite{state-free-vampnets} for molecular kinetics.

  In contrast to the transfer operator and the Koopman operator, the infinitesimal generator (henceforth called generator) is
  a key operator of a molecular system that does not rely on the choice of the lag time. Similar to the
  aforementioned operator-based approaches using the transfer operator or the
  Koopman operator, crucial information on the underlying dynamics can be
  obtained by analyzing the system's generator as well. For instance, the
  leading eigenvalues of the generator encode the dominant timescales of
  metastable molecular systems, whereas the corresponding eigenfunctions are
  in fact good collective variables for constructing low-dimensional
  coarse-grained dynamics~\cite{effective_dyn_2017}. In a broader context,
  many probabilistic quantities of a Markov process can be represented as the solution to certain
  partial differential equation (PDE) that involves the system's generator~\cite{oksendalSDE,lelievre_stoltz_2016}. This fact has inspired the PDE-based approaches, which have been successfully applied in analyzing
  various aspects of Markov processes, such as metastability~\cite{metastability-bovier,LeBrisLelievreLuskinPerez+2012+119+146},
  transition paths~\cite{tpt_eric}, and more recently the model reduction of molecular
  dynamics~\cite{effective_dyn_2017}. Moreover, data-based methods for estimating the eigenvalues and eigenfunctions of the generator are available~\cite{sz-entropy-2017}. 

  The overdamped Langevin dynamics~\cite{oksendalSDE,Pavliotis2014} is often
  adopted in modelling molecular dynamics in equilibrium due to its simplicity and nice properties for mathematical analysis. In this context, we consider a smooth potential function $V: \mathbb{R}^d \rightarrow \mathbb{R}$ in state space~$\mathbb{R}^d$, a matrix-valued
  function $\sigma: \mathbb{R}^d\rightarrow \mathbb{R}^{d\times d_1}$, where
  $d, d_1$ are two integers such that $1 \le d \le d_1$, and we define the
  function $a: \mathbb{R}^d\rightarrow \mathbb{R}^{d\times d}$ by~$a =
  \sigma\sigma^T$. The entries of the matrix $a(x)$ at $x \in \mathbb{R}^d$ are $a_{ij}(x)$, where $1 \le i, j \le d$. Then, the overdamped Langevin dynamics is described by the stochastic differential equation (SDE)
  \begin{align}
    dx(s) = - a(x(s)) \nabla V(x(s))\,ds + 
    \frac{1}{\beta} (\nabla\cdot a)(x(s))\,ds +  
    \sqrt{2\beta^{-1}} \sigma(x(s))\, dw(s)
    \label{overdamped}
  \end{align}
   where $x(s)\in \mathbb{R}^d$ is the state of the system at time $s\in [0,
   +\infty)$, ${\beta > 0}$ is a constant proportional to the inverse of the
  system's temperature, $(w(s))_{s\ge 0}$ is a Brownian motion in
  $\mathbb{R}^{d_1}$, and $\nabla\cdot a: \mathbb{R}^d\rightarrow \mathbb{R}^d$
  denotes the vector-valued function whose components are given by $(\nabla\cdot a)_i(x) =
  \sum_{j=1}^d\frac{\partial a_{ij}}{\partial x_j}(x)$ for $i=1,\ldots, d$ and for all $x\in \mathbb{R}^d$.  
  The reason for including the term $\frac{1}{\beta} \nabla\cdot a$ in
  \eqref{overdamped} is to make sure that under certain assumptions (see Section~\ref{subsec-math-setting}) 
   the process \eqref{overdamped} is ergodic with respect to the unique invariant probability distribution~$\mu$, defined by
  \begin{align}
    \mu(dx) = \frac{1}{Z} e^{-\beta V(x)}\,dx\,,\quad ~x\in \mathbb{R}^d\,,   \label{mu-invariant}
  \end{align}
  where $Z=\int_{\mathbb{R}^d} e^{-\beta V(x)}\,dx$ is the normalizing constant.
  We refer to \cite[Section~5]{down1995}, \cite[Theorem~5.3]{Mattingly2002}, and~\cite[Theorem 1]{DuncanLelievrePavliotis16} 
  for sufficient conditions on ergodicity and to~\cite{oksendalSDE,Pavliotis2014} for a more detailed introduction to the SDE~\eqref{overdamped}.

  The main focus of this paper is the eigenvalue problem defined by the (high-dimensional) PDE 
  \begin{align}
    -\mathcal{L} \varphi = \lambda \varphi\,,\quad \mbox{in}~~\mathbb{R}^d
    \label{eigen-problem-rd}
  \end{align}
  associated to the generator $\mathcal{L}$ of SDE~\eqref{overdamped}, given by 
  \begin{align}
    \begin{split}
    \mathcal{L}f =& \frac{\mathrm{e}^{\beta V}}{\beta} \sum_{i,j=1}^d\frac{\partial}{\partial
      x_j}\left(\mathrm{e}^{-\beta V} a_{ij}\frac{\partial f}{\partial x_i} \right) \,,
    \end{split}
    \label{generator-l}
  \end{align}
  where $f:\mathbb{R}^d \rightarrow \mathbb{R}$ is a test function. 
  In particular, when $a$ is the identity matrix of size $d$, the generator $\mathcal{L}$ in \eqref{generator-l} has the well-known form
  \begin{align}
    \begin{split}
      \mathcal{L}f =& -\nabla V \cdot \nabla f + \frac{1}{\beta} \Delta f \,.
    \end{split}
    \label{generator-l-identity}
  \end{align}
Under mild conditions (see Section~\ref{subsec-math-setting}), the operator $\mathcal{L}$
  is self-adjoint with respect to a weighted inner product and it has purely
  discrete spectrum. Moreover, one can show that the eigenvalues of the
  problem \eqref{eigen-problem-rd} are all positive real numbers, except the
  trivial one ${\lambda_0=0}$ whose corresponding eigenfunction is
  $\varphi_0\equiv 1$. In this paper, we are interested in computing the first
  non-trivial $K$ eigenvalues (in non-decreasing order) and the corresponding eigenfunctions of \eqref{eigen-problem-rd} for some integer $K \ge 1$, i.e., the eigenpairs $\{(\lambda_i, \varphi_i)\}_{1\le i \le K}$, where 
  \begin{equation}
    0 = \lambda_0 < \lambda_1 \le \lambda_2 \le \dots \le \lambda_K \le \dots\,.
    \label{sequence-of-lambda}
  \end{equation}
Given $K \ge 1$, the main contribution of this paper is a novel numerical
method for computing the eigenpairs $\{(\lambda_i, \varphi_i)\}_{1\le i \le
K}$ of \eqref{eigen-problem-rd} by training artificial neural networks. 
Designing the loss function based on certain variational formulation of the
eigenpairs of \eqref{eigen-problem-rd}, we propose a training task which
computes multiple eigenpairs of the problem~\eqref{eigen-problem-rd} at once.
The method can be applied to solving high-dimensional eigenvalue PDEs (i.e.,
$d$ is large) where the operator $\mathcal{L}$ is of the general
form~\eqref{generator-l}. To overcome the metastability in sampling the
training data according to $\mu$, a reweighting technique is proposed, which
allows the use of biased training data sampled from a distribution other than $\mu$.

In the following let us mention several existing work on
related topics and discuss the novelty of the current work. Firstly,
the aforementioned numerical approaches based on the transfer operator or
the Koopman operator (see the discussion at the beginning of this
introduction) require a careful choice of the lag time~\cite{msm_generation}. 
In contrast, our method based on the generator does not involve the lag time. The
current work extends the data-based method using basis functions in~\cite{sz-entropy-2017} to the neural network framework. As a result, the algorithm in this work does not require the choice of basis functions, and therefore is
suitable for solving high-dimensional problems. We note that similar biased sampling
and reweighting technique have been employed
in~\cite{chasing-cv-autoencodes-biased-traj,committor-by-nn}. 
Secondly, beyond the study of dynamical systems, there has been growing research interest in recent years in developing deep learning-based numerical methods for solving high-dimensional PDEs, such as parabolic PDEs~\cite{Han-pde-pnas2018}, the committor equation (a boundary value elliptic PDE) in molecular dynamics~\cite{committor-by-nn}, and the eigenvalue PDEs~\cite{deep-ritz-weinan,HAN2020-eigenpde}. We also refer to the recent work~\cite{schroedinger-eigenpde,HAN2019-solve-many-electron-schroedinger-equation,dnn-electron-schroedinger-noe,ab-initio-many-electron-schrodinger-dnn}
for deep learning-based methods for solving eigenvalue problems in quantum
physics. In contrast to these existing methods for eigenvalue PDEs which
typically compute the first eigenvalue of the problem, our algorithm is able
to find multiple eigenpairs by solving a single training task. Lastly, we note that finding good
low-dimensional collective variables for complex molecular systems is of great
importance in the study of molecular kinetics. Various approaches are
proposed, using for instance the eigenfunctions in the operator-based
approaches~\cite{tica}, feature
engineering~\cite{peters-trout-2006,sultan-kiss-shukla-pande-2014,sultan-pande-automate-design-cv},
and
autoencoders~\cite{chen-tan-ferguson-cv-discovery-2018,chen-ferguson-on-the-fly-cv-discovery-2018,hythem-chen-ferguson-2020,chasing-cv-autoencodes-biased-traj}.
Although this topic is not the focus of the current paper, we point out that
it is potentially possible to reformulate our (training) algorithm as an algorithm for finding good collective variables of molecular systems. The application of the current work in this direction will be explored in future work.

The remainder of this article is organized as follows. In
Section~\ref{sec-theoretical-background}, we present the mathematical setting
of the eigenvalue PDE problem considered in this paper, and in particular we study its variational formulations. 
In Section~\ref{sec-algo-consideration}, we present the numerical algorithm for
solving the eigenvalue PDE problem by training artificial neural networks.
In Section~\ref{sec-example}, we demonstrate our algorithm by studying two numerical examples.
In Section~\ref{sec-discuss}, we conclude with discussions on various issues.
Appendix~\ref{sec-app-proof} contains the proofs of two results in Section~\ref{sec-theoretical-background}.

  \section{Mathematical setting}
  \label{sec-theoretical-background}
  In this section we present mathematical background of the eigenvalue problem \eqref{eigen-problem-rd}. In particular, we provide variational formulations for the leading eigenpairs of the eigenvalue problem \eqref{eigen-problem-rd}. 
  \subsection{Preliminary}
  \label{subsec-math-setting}
Throughout this paper, we make the following assumption on the function $a$ in \eqref{generator-l}.
  \begin{assump}
    The function $a: \mathbb{R}^d\rightarrow \mathbb{R}^{d\times d}$ is $C^2$-smooth and satisfies the (uniform ellipticity) condition 
    \begin{align}
    v^T a(x) v  \ge a^- |v|^2 \,, \quad \forall~x, v \in \mathbb{R}^d\,,
    \label{elliptic-a}
  \end{align}
  for some positive constant $a^->0$.
    \label{assump-a}
  \end{assump}
We denote by $\mathbb{N}:=\{1,2,\dots\}$ the set of natural numbers.
The mean value of a measurable function $f: \mathbb{R}^d\rightarrow \mathbb{R}$ with respect to the invariant probability
measure $\mu$ defined in \eqref{mu-invariant} is denoted as $\mathbf{E}_\mu(f)$ provided that it exists, i.e.,\
  \begin{equation}
    \mathbf{E}_\mu (f) := \int_{\mathbb{R}^d} f(x)\,\mu(dx)\,.
    \label{mean-wrt-mu}
  \end{equation}
The Hilbert space $L^2(\mu)$ consists of all measurable functions that are square-integrable with respect to $\mu$, 
with the norm and the inner product defined by, respectively,
\begin{equation}
  \|f\|_\mu := \mathbf{E}_\mu(f^2)^{\frac{1}{2}}, \quad \langle f, g\rangle_\mu :=
  \mathbf{E}_\mu(fg)\,, \quad \forall~ f, g \in L^2(\mu)\,.
\end{equation}
Also recall the following expressions of the variance and the covariance of functions with respect to $\mu$ : for $f,g \in L^2(\mu)$,
\begin{equation}
  \Var(f) = \mathbf{E}_\mu \big(f^2 \big) - (\mathbf{E}_\mu (f))^2 ,\quad 
  \Cov(f,g) = \mathbf{E}_\mu \big(fg \big) - \mathbf{E}_\mu(f) \mathbf{E}_\mu (g) \,. 
  \label{def-var-cov}
\end{equation}
Clearly, we have $\Cov(f,f)=\Var(f)$ for all $f \in L^2(\mu)$. For the
operator $\mathcal{L}$ \eqref{generator-l}, using \eqref{mu-invariant} and the integration by parts formula, we can verify that 
  \begin{equation}
    \begin{aligned}
       \langle (-\mathcal{L}) f, g\rangle_\mu = \langle f, (-\mathcal{L}) g\rangle_\mu = \frac{1}{\beta} \mathbf{E}_\mu \big((a\nabla f)\cdot \nabla g\big) \,,
    \end{aligned}
    \label{lfg-glf}
  \end{equation}
  for all $C^2$ test functions $f, g$ such that the integrals in \eqref{lfg-glf} are well-defined.

In the following we collect some useful results from~\cite{DuncanLelievrePavliotis16}. 
We need the following assumption \cite[Assumptions B and C]{DuncanLelievrePavliotis16} on $V$.
  \begin{assump}
    The function $V\in C^2(\mathbb{R}^d)$ is bounded from below and satisfies:
    \begin{enumerate}
      \item
    There exists $\delta \in (0,1)$, such that 
    \begin{equation}
      \liminf_{|x|\rightarrow +\infty} \big((1-\delta) \beta|\nabla V(x)|^2 - \Delta V(x)\big)  > 0 \,;
    \end{equation}
  \item
      $\lim\limits_{|x|\rightarrow +\infty} |\nabla V(x)| = +\infty$.
    \end{enumerate}
    \label{assump-on-v}
  \end{assump}

Define the space $L^2_0(\mu):=\{f\in L^2(\mu)\,|\, \mathbf{E}_\mu (f)=0\}$.
Under Assumptions~\ref{assump-a}--\ref{assump-on-v}, using \eqref{lfg-glf} and
the inequality \eqref{poincare-ineq} in Lemma~\ref{lemma-poisson} below we can show that $\|\cdot\|_1$, given by 
  \begin{equation}
    \|f\|_1:= \langle f, (-\mathcal{L}) f\rangle_\mu^{\frac{1}{2}}
    \label{l1-seminorm}
  \end{equation}
  for a test function $f$, defines a norm in the space
  \begin{equation}
    \mathcal{H}^1 := \Big\{f\in L^2(\mu)\,\Big|\, \mathbf{E}_\mu(f)=0,\, \|f\|_1 < +\infty\Big\}
    \label{def-space-h1}
  \end{equation}
and it satisfies the Pythagorean identity. Therefore, the completion of $\mathcal{H}^1$ with respect to 
  $\|\cdot\|_1$ \eqref{l1-seminorm} is a Hilbert space, which we again denote by $\mathcal{H}^1$.
  The inner product $\langle f,g\rangle_1$ of $\mathcal{H}^1$, defined through polarization, is actually given by
  \eqref{lfg-glf} for all test functions $f, g \in \mathcal{H}^1$.
  A detailed analysis of the space $\mathcal{H}^{1}$ can be found in~\cite{olla:hal-00722537}. 

  In view of the last expression in \eqref{lfg-glf}, we define the energy 
 $\mathcal{E}:L^2_0(\mu)\rightarrow [0, +\infty]$ as  
  \begin{align}
    \mathcal{E}(f):=\left.
      \begin{cases} 
	 \frac{1}{\beta} \mathbf{E}_\mu \big((a\nabla f)\cdot \nabla f\big)\,, & f \in \mathcal{H}^1\,,\\
	+\infty  & f \in L^2_0(\mu) \setminus \mathcal{H}^1\,.
      \end{cases}
      \right.
    \label{energy}
  \end{align}
 The operator $\mathcal{L}$ can be extended to a self-adjoint operator on
  $L^2_0(\mu)$, with the domain $D(\mathcal{L})=\{\psi \in L^2_0(\mu)\,|\,
  \exists f\in L^2_0(\mu), \mathcal{L}\psi=f\}$.
  By Cauchy-Schwarz inequality, it is straightforward to verify that
  $\|\psi\|_1 < +\infty$ for all $\psi\in D(\mathcal{L})$, from which we
  conclude that $D(\mathcal{L})\subset \mathcal{H}^1$.

  Assumptions~\ref{assump-a} and \ref{assump-on-v} are sufficient to guarantee the compactness of the embedding 
  $\mathcal{H}^1 \hookrightarrow L_0^2(\mu)$, as stated in Lemma~\ref{lemma-poisson} below.
\begin{lemma}[{{\cite[Lemma 2]{DuncanLelievrePavliotis16}}}]
Suppose that Assumptions \ref{assump-a} and \ref{assump-on-v} hold. Then the
  embedding $\mathcal{H}^1 \hookrightarrow L_0^2(\mu)$ is compact. The operator $\mathcal{L}$
  satisfies the Poincar{\'e} inequality:
  \begin{equation}
    \lambda\|g\|_\mu^2 \le \langle g, (-\mathcal{L}) g\rangle_\mu, \quad \forall g \in \mathcal{H}^1\,,
    \label{poincare-ineq}
  \end{equation}
  where $\lambda$ is a positive constant. Moreover, for all $f\in L_0^2(\mu)$,
  there exists a unique $\psi\in \mathcal{H}^1$ such that
  $-\mathcal{L}\psi=f$.
  \label{lemma-poisson}
\end{lemma}
Note that Lemma~\ref{lemma-poisson} implies that the operator $-\mathcal{L}: D(\mathcal{L})\rightarrow L^2_0(\mu)$ is bijective and therefore admits a unique inverse $(-\mathcal{L})^{-1}$.

  In the following we consider the spectrum of $-\mathcal{L}$. We show that
  $-\mathcal{L}$ has purely discrete spectrum under Assumptions~\ref{assump-a} and~\ref{assump-on-v}. Let us first recall some definitions.  Denote by $I$ the identity operator on $L^2_0(\mu)$. The spectrum of~$-\mathcal{L}$, denoted by $\sigma(-\mathcal{L})$, is the set consisting of all complex values $\lambda \in
  \mathbb{C}$, for which the operator $-\mathcal{L} - \lambda I: D(-\mathcal{L})\subset L^2_0(\mu) \rightarrow L^2_0(\mu)$ 
   does not have a bounded inverse.
  The self-adjointness of $-\mathcal{L}$ implies that $\sigma(-\mathcal{L})\subset \mathbb{R}$. 
  In this case, the discrete spectrum of $-\mathcal{L}$, denoted by
  $\sigma_{disc}(-\mathcal{L})$, is the subset of~$\sigma(-\mathcal{L})$ consisting of isolated eigenvalues $\lambda\in\sigma(-\mathcal{L})$ with
  finite multiplicity, i.e.,\ $\{\lambda'\in \mathbb{C}\,|\,
  \lambda'\in\sigma(-\mathcal{L}), |\lambda'-\lambda| < \epsilon\} = \{\lambda\}$ for some $\epsilon > 0$ and the eigenspace $\{\psi\in
  L_0^2(\mu)\,|\, -\mathcal{L}\psi = \lambda \psi\}$ has finite (non-zero)
  dimension. We say that $-\mathcal{L}$ has purely discrete spectrum if
  $\sigma(-\mathcal{L})=\sigma_{disc}(-\mathcal{L})$. See~\cite[Section 2.4]{teschl2009mathematical} and~\cite[Chapter VII and Section
  VIII.3]{reed1981functional} for careful studies on the spectrum of
  self-adjoint operators in Hilbert spaces.

Applying Lemma~\ref{lemma-poisson}, we obtain the results below which guarantee the
compactness of $(-\mathcal{L})^{-1}$ and the fact that $-\mathcal{L}$ has purely discrete spectrum. 
Its proof is presented in Appendix~\ref{sec-app-proof}.
  \begin{prop}
The following two results hold under Assumptions \ref{assump-a} and \ref{assump-on-v}. 
    \begin{enumerate}
      \item
    The operator $(-\mathcal{L})^{-1}: L^2_0(\mu)\rightarrow L^2_0(\mu)$ is compact.
      \item
	There exist an orthonormal basis $(\varphi_i)_{i\ge 1}$ in
	$D(\mathcal{L})$ and a sequence of positive numbers $(\lambda_i)_{i\ge 1}$, 
	where $0 < \lambda_1\le \lambda_2 \le \cdots$ and $\lim_{i\rightarrow +\infty} \lambda_i =+\infty$,
	such that $-\mathcal{L}\varphi_i=\lambda_i\varphi_i$ for $i \ge 1$.
	Moreover, we have 
    \begin{equation}
      \sigma(-\mathcal{L})=\sigma_{disc}(-\mathcal{L}) = \{\lambda_1,
      \lambda_2, \lambda_3, \dots\}\,.
      \end{equation}
    \end{enumerate}
    \label{prop-spectrum}
  \end{prop}

  \subsection{Variational characterization}
  \label{subsec-variational}
  In this section, we present a variational characterization of the first $K$ eigenpairs $\{(\lambda_i, \varphi_i)\}_{1\le i \le K}$ of \eqref{eigen-problem-rd}, where $K \in \mathbb{N}$. Note that by Proposition~\ref{prop-spectrum} we can assume without loss of generality that the eigenfunctions $(\varphi_i)_{1 \le i \le K}$ are both normalized and pairwise orthogonal.

  First, let us recall the min-max theorem for positive definite operators~\cite[Section 12.1]{lieb2001analysis}, i.e.,\ 
  \begin{equation}
  \begin{aligned}
    \lambda_k = \min_{H_k} \max_{f\in H_k,\, \|f\|_\mu=1} \mathcal{E}(f) \,, \quad k \ge 1\,,
  \end{aligned}
    \label{lambda-u-variational-ith}
  \end{equation}
where $\lambda_k$ is the $k$th eigenvalue  of \eqref{eigen-problem-rd} in \eqref{sequence-of-lambda}, 
$\mathcal{E}(\cdot)$ is the energy in \eqref{energy}, and the minimum is over all $k$-dimensional
subspaces $H_k$ of $\mathcal{H}^1$.  In particular,
\eqref{lambda-u-variational-ith} with $k=1$ implies that the first eigenpair $(\lambda_1, \varphi_1)$ solves
  \begin{align}
    \lambda_1 = \mathcal{E}(\varphi_1)\,,~ \mbox{where}~ \varphi_1\in \argmin_{f\in \mathcal{H}^1,\, \|f\|_\mu=1} \mathcal{E}(f)\,.
    \label{lambda1-u1-variational}
  \end{align}

  To present the variational characterization of the first $K$ eigenpairs, let
  us define the matrix for $k\in \mathbb{N}$
\begin{equation}
  \begin{aligned}
    &F^{(k)}(f_1, f_2, \dots, f_k) = \Big(F^{(k)}_{jj'}(f_1, f_2, \dots, f_k)\Big)_{1 \le j,j'\le k}\in \mathbb{R}^{k\times k}, \\
   \mbox{where}\quad & F^{(k)}_{jj'}(f_1, f_2, \dots, f_k) = \frac{1}{\beta} \int_{\mathbb{R}^d} (a\nabla f_j)\cdot \nabla f_{j'}\,d\mu\,,
    \end{aligned}
  \label{def-f-k}
\end{equation}
for functions $f_1, f_2, \dots, f_k\in \mathcal{H}^1$.
The main result of this section is then the following variational characterization of the first $K$ eigenpairs $\{(\lambda_i, \varphi_i)\}_{1\le i \le K}$.
\begin{thm}
Let $K\in \mathbb{N}$. Suppose that Assumptions \ref{assump-a} and \ref{assump-on-v} hold. 
  Define the diagonal matrix
  \begin{equation}
\Sigma:=\mbox{\textnormal{diag}}\{\omega_1, \omega_2, \dots, \omega_K\}\in \mathbb{R}^{K\times K}\,,
    \label{mat-sigma}
  \end{equation}
  where $(\omega_i)_{1\le i \le K}$ is a non-increasing sequence of positive numbers, i.e.,\ $\omega_1 \ge \dots \ge \omega_K >0$. Then, we have 
  \begin{align}
      \sum_{i=1}^K \omega_i\lambda_i =\min_{f_1,\dots, f_K\in \mathcal{H}^1} \sum_{i=1}^K \omega_i \mathcal{E}(f_i) 
      = \min_{f_1,\dots, f_K\in \mathcal{H}^1} \mbox{\textnormal{tr}} \Big(\Sigma F^{(K)}(f_1,f_2,\dots, f_K)\Big)\,,
    \label{variational-all-first-k}
  \end{align}
   where $F^{(K)}(f_1,f_2,\dots, f_K)$ is the $K\times K$ matrix defined in \eqref{def-f-k} (with $k=K$), and the minimum is over all $f_1,f_2,\dots, f_K\in \mathcal{H}^1$ such that 
  \begin{equation} 
   \langle f_i,f_j\rangle_\mu = \delta_{ij}\,, \quad \forall i,j\, \in \{1,\dots,K\}\,.
    \label{f-orthonormal}
  \end{equation}
 Moreover, the minimum in \eqref{variational-all-first-k} is achieved when $f_i=\varphi_i$ for $1 \le i \le K$.
  \label{thm-variational-form}
\end{thm}
Theorem~\ref{thm-variational-form} was obtained in~\cite{sz-entropy-2017}
using calculus of variations. In Appendix~\ref{sec-app-proof}, we present an
alternative proof of Theorem~\ref{thm-variational-form} by applying Ruhe's trace inequality~\cite{Ruhe1970,marshall-inequalities-book}. 

Note that \eqref{variational-all-first-k} reduces to the min-max theorem~\eqref{lambda1-u1-variational} when $K=1$.
In the general case, the characterization \eqref{variational-all-first-k}-\eqref{f-orthonormal}
allows us to develop numerical methods for computing multiple eigenpairs of
$-\mathcal{L}$ by solving a single optimization problem (see Section~\ref{sec-algo-consideration}).

We conclude this section with a remark on other types of variational formulations.
    \begin{remark}
       Denote by $\lambda_{\max}\big(F^{(k)}(f_1,f_2, \dots, f_k)\big)$ the maximum
      eigenvalue of the matrix $F^{(k)}(f_1,f_2, \dots, f_k) \in
      \mathbb{R}^{k\times k}$ in \eqref{def-f-k}, where $k \in \mathbb{N}$. 
By considering an orthonormal basis and applying the min-max principle for
      symmetric matrices to $\lambda_{\max}\big(F^{(k)}(f_1,f_2, \dots,
      f_k)\big)$, we can obtain the following equivalent formulation of \eqref{lambda-u-variational-ith}: 
  \begin{equation}
    \begin{aligned}
      \lambda_k &= \min_{f_1,f_2,\dots, f_k\in \mathcal{H}^1}\max_{c\,\in \mathbb{R}^k, |c|=1}
    \mathcal{E}\Big(\sum_{i=1}^k c_i f_i\Big) \\
      & = \min_{f_1,f_2,\dots, f_k\in \mathcal{H}^1}
    \lambda_{\max}\big(F^{(k)}(f_1,f_2, \dots, f_k)\big) \,,
    \end{aligned}
    \label{lambda-u-variational-kth-basis-function}
  \end{equation}
	where the minimum is over all $(f_i)_{1\le i \le k}\subset
	\mathcal{H}^1$ such that $\langle f_{i}, f_{j}\rangle_\mu =
	\delta_{ij}$ for all~$i, j \in \{1,\dots, k\}$. 
	Besides, the eigenvalues satisfy the max-min principle~\cite[Section 12.1]{lieb2001analysis}: 
  \begin{equation}
    \lambda_k = \max_{f_1, f_2, \dots, f_{k-1}\in \mathcal{H}^1}\min_{f \in
    H_{k-1}^{\perp}, \|f\|_\mu=1} \mathcal{E}(f) \,, \quad \forall~ k \ge 1\,,
    \label{lambda-u-variational-max-min}
  \end{equation}
where $H_{0}^{\perp}=\mathcal{H}^1$ and, for fixed $(f_i)_{1 \le i \le {k-1}} \in \mathcal{H}^1$ with $k>1$, 
	$H_{k-1}^{\perp}:=\{f\in \mathcal{H}^1\,|\, \langle f, f_i\rangle_\mu=0, \forall~ 1 \le i \le {k-1}\}$. 
	Note that, in contrast to \eqref{lambda-u-variational-kth-basis-function}, 
in \eqref{lambda-u-variational-max-min} the functions $(f_j)_{1\le j\le i-1}$ do not have to be linearly independent. 
      Also, both \eqref{lambda-u-variational-kth-basis-function} and \eqref{lambda-u-variational-max-min} recover \eqref{lambda1-u1-variational} when $k=1$.

	While in this paper we propose numerical algorithms based on the
	variational formulation \eqref{variational-all-first-k}--\eqref{f-orthonormal}, let us point out that it is also
	possible to develop numerical algorithms for computing the eigenpairs of $-\mathcal{L}$
	based on \eqref{lambda-u-variational-kth-basis-function} or \eqref{lambda-u-variational-max-min}.
  \end{remark}

  \subsection{Neural network spaces} 
  \label{subsec-neural-network-space}
  In this section we introduce the neural network spaces. For brevity
  we only consider feedforward neural networks
  following~\cite{PETERSEN2018296}. However, we point out that both the
  discussion and the numerical algorithm in this paper, i.e., Algorithm~\ref{algo-1} in Section~\ref{sec-algo-consideration}, can be directly adapted to more general types of neural networks. 
  
  Given $L, N_0, N_1, \dots, N_L \in \mathbb{N}$, 
  the space $\mathcal{S}$ of $L$-layer neural networks with the architecture
\begin{equation}  
  \mathcal{N}:= (N_0, N_1, \dots, N_L) 
  \label{def-arch-nn}
\end{equation}  
  is defined as the set of sequences of matrix-vector tuples
  \begin{equation}
    \begin{aligned}
      \mathcal{S} := \Big\{\big((A_\ell, b_\ell)\big)_{1\le \ell \le L} \,\Big|\,
      A_\ell\in \mathbb{R}^{N_\ell \times N_{\ell-1}},\, b_\ell \in \mathbb{R}^{N_l},\, 
        \ell\in \{1,2,\dots,L\}\Big\}\,.
    \end{aligned}
    \label{nn-space}
  \end{equation}
For each $\Phi\in \mathcal{S}$, there are $N_\ell$ neurons in the $\ell{\mathrm{th}}$ layer of $\Phi$, where $\ell\in \{0,1,\dots,L\}$. 
  These layers are called the input layer for $\ell=0$, the output layer for
  $\ell=L$, and the hidden layer for $1 \le\ell <L$, respectively. 
See Figure~\ref{fig-nn} for the illustration of neural networks with $L=4$ layers.
  Note that the space $\mathcal{S}$ can be viewed as the Euclidean space
  $\mathbb{R}^N$, where the dimension $N = \sum_{\ell=1}^L N_\ell(N_{\ell-1} +
  1)$ is equal to the total number of parameters in a neural network $\Phi\in \mathcal{S}$.

     Next, we discuss the use of neural networks in representing functions.
     Let $\rho: \mathbb{R}\rightarrow \mathbb{R}$ be a $C^1$-smooth activation function.
 Given a neural network $\Phi =((A_\ell, b_\ell))_{1 \le \ell \le L} \in
 \mathcal{S}$ with the architecture $\mathcal{N}$ \eqref{def-arch-nn}, the
 realization of $\Phi$ is defined as the function $\RmapNN(\Phi):=f\in C(\mathbb{R}^{N_0}, \mathbb{R}^{N_L})$, whose value $f(x) = y\in \mathbb{R}^{N_L}$ at any $x\in \mathbb{R}^{N_0}$ is determined through the following scheme: 
      \begin{equation}
	\begin{aligned}
	  h^{(0)} &:= x, \\
	  h^{(\ell)} &:= \rho(A_\ell h^{(\ell-1)} + b_\ell), \quad \forall~\ell \in \{1,2, \dots, L-1\}\,, \\
	  y &:= A_{L} h^{(L-1)} + b_{L} \,.
	\end{aligned}
	\label{x-to-y}
      \end{equation}
The map 
\begin{equation}
\RmapNN: \mathcal{S} \rightarrow C(\mathbb{R}^{N_0}, \mathbb{R}^{N_L})
  \label{realization-map}
\end{equation}
is called the realization map. Note that, with slight abuse of notation, the
action of $\rho$ on vectors in \eqref{x-to-y} is defined componentwise, i.e.,\ for $\ell\in \{1,2,\dots,
      L-1\}$,
      \begin{equation*}
	\rho(h) := \big(\rho(h_1), \rho(h_2),
      \ldots, \rho(h_{N_l})\big)^T\in \mathbb{R}^{N_\ell}\,, \quad \forall~h = (h_1, h_2,\ldots,
	h_{N_\ell})^T \in \mathbb{R}^{N_\ell}\,.
      \end{equation*}
Also, for the sake of notational simplicity, we have omitted the dependence of $\mathcal{R}$ on the activation
function $\rho$, since the latter is assumed fixed once it is chosen.

  \begin{figure}[h!]
  \centering
  \includegraphics[width=0.70\textwidth]{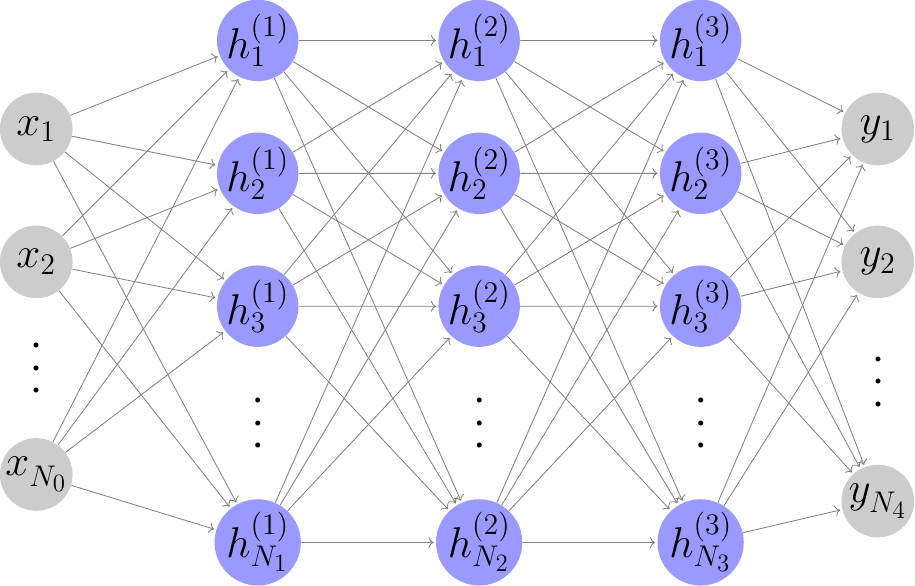}
    \caption{Illustration of neural networks with $L=4$ layers (i.e.,\ three hidden layers).
  \label{fig-nn}}
\end{figure}

  \section{Numerical algorithm}
  \label{sec-algo-consideration}
Based on the variational formulation \eqref{variational-all-first-k}--\eqref{f-orthonormal}, in this section we
propose a numerical algorithm for solving the eigenpairs $((\lambda_i,
\varphi_i))_{1 \le i \le K}$ of the PDE problem \eqref{eigen-problem-rd} by training
artificial neural networks introduced in
Section~\ref{subsec-neural-network-space}. Before presenting the algorithm,
let us first discuss the numerical treatment of both the constraints and the (high-dimensional) integrals in \eqref{variational-all-first-k}--\eqref{f-orthonormal}.

First, we consider the constraints involved in \eqref{variational-all-first-k}--\eqref{f-orthonormal}. Note that the minimization
\eqref{variational-all-first-k} is over functions with zero means (see the
definition of~$\mathcal{H}^1$ in \eqref{def-space-h1}) 
such that the pairwise orthonormality condition~\eqref{f-orthonormal} holds.
In practice, given $f\in L^2(\mu)$, the zero mean condition can be easily
imposed by applying a shift operation to $f$, that is, by considering the function $f-\mathbf{E}_\mu f$. 
For functions $f_1, f_2, \cdots, f_K \in L^2(\mu)$, we denote 
\begin{equation}
f^{\mathrm{center}}_i:=f_i-\mathbf{E}_\mu f_i\,,\quad 1 \le i \le K\,.
  \label{shift-f-to-center}
\end{equation}
Then, concerning the pairwise orthonormality condition~\eqref{f-orthonormal}, using the definition~\eqref{def-var-cov} we find that the following two conditions are equivalent:
\begin{equation}
\Cov(f_i, f_j)=\delta_{ij},\, \forall~1 \le i,j \le K \quad
  \Longleftrightarrow \quad (f^{\mathrm{center}}_i)_{1\le i\le K}~ \mathrm{satisfy\, \eqref{f-orthonormal}}\,. 
  \label{constraints-by-cov}
\end{equation}
In other words, once $(f_i)_{1 \le i \le K}$ satisfy the condition on the left hand side of 
\eqref{constraints-by-cov}, the functions $(f_i^{(\mathrm{center}})_{1 \le i \le K}$ satisfy both the zero mean condition and the pairwise orthonormality condition~\eqref{f-orthonormal}. As we will see below, this
fact allows us to work with the unshifted functions $(f_i)_{1 \le i \le K}$
when we train neural networks in the algorithm. The
eigenfunctions with zero means can be easily recovered according to \eqref{shift-f-to-center}.

Next, we consider the approximation of the integrals in \eqref{variational-all-first-k}--\eqref{f-orthonormal}, or the mathematical expectation \eqref{mean-wrt-mu} in a general form. We discuss estimators based on a reweighting
technique that is useful in alleviating sampling difficulties by allowing the use of biased sample data
(see~\cite{chasing-cv-autoencodes-biased-traj,committor-by-nn} for similar
techniques).  Let $\bar{\mu}$ be a probability measure in $\mathbb{R}^d$ such that $\mu$ \eqref{mu-invariant} is absolutely continuous with respect to $\bar{\mu}$. By a change of measures, we have 
  \begin{equation}
    \mathbf{E}_\mu(f) = \int_{\mathbb{R}^d} f(x)\,\mu(dx) = \int_{\mathbb{R}^d}
    f(x)\,\frac{d\mu}{d\bar{\mu}}(x)\, \bar{\mu}(dx) =
    \mathbf{E}_{\bar{\mu}} \Big(f\,\frac{d\mu}{d\bar{\mu}}\Big)\,,
    \label{is-formula}
  \end{equation}
for a measurable test function $f: \mathbb{R}^d\rightarrow \mathbb{R}$, where
$\mathbf{E}_{\bar{\mu}}$ denotes the expectation with respect to $\bar{\mu}$.
  Assume that $n$ states $(x^{(\ell)})_{1\le \ell \le n} \subset \mathbb{R}^d$
  are sampled according to $\bar{\mu}$, where $n \in \mathbb{N}$.  Then, based on \eqref{is-formula}, we can approximate the mean $\mathbf{E}_\mu(f)$ by
  the importance sampling estimator 
  \begin{equation}
    \mathbf{E}^{\textrm{data}}\,(f) := \frac{\sum_{\ell=1}^n f(x^{(\ell)})\,\upsilon_\ell}{\sum_{\ell=1}^n \upsilon_\ell}
\,, \quad \mbox{\textnormal{where}}~ \upsilon_\ell=
    \frac{d\mu}{d\bar{\mu}}(x^{(\ell)})\,, \quad \forall~ \ell \in \{1,2,\dots, n\}\,.
    \label{mean-wts-mu-empirical}
  \end{equation}
Typically, we choose $\bar{\mu}$ such that comparing to $\mu$ it
  is easier to sample states according to $\bar{\mu}$ (e.g.,\ less
  metastability in $\bar{\mu}$). At the same time, $\bar{\mu}$
  should not be too different from $\mu$, so that the weights
  $(\upsilon_\ell)_{1\le \ell \le n}$ in \eqref{mean-wts-mu-empirical} 
  are neither too small nor too large. One can simply use $\bar{\mu}=\mu$
  (i.e.,\ $\upsilon_\ell=1.0$) when sampling directly from $\mu$ is not a difficult task.
For the computation of the weights in practice, we refer to the discussions in
the numerical example in Section~\ref{subsec-numerical-model-problem}, in the setting where the states are sampled from $\bar{\mu}=\mu$, and to the detailed discussions in the numerical example in Section~\ref{subsec-numerical-ad} (see \eqref{weights-v-ad}), in the setting where the states are sampled from a biased simulation.

In practice, a subset of the sample data $(x^{(\ell)})_{1\le \ell \le n}$ (i.e.,\ mini-batch) is often used in
training neural networks. Corresponding to this setting, let us consider a
sequence of indices~\footnote{Precisely, $\mathcal{I}$ is a multiset, since
the repetition of indices is allowed and their ordering is unimportant.}
\begin{equation}
  \mathcal{I}=(\ell_i)_{1\le i \le B}\,, \quad \mbox{where}~ \ell_i\in \{1,2,\dots, n\}\,,  
  \label{index-subset}
\end{equation} 
for some $B \in \mathbb{N}$ and $1 \le B \le n$. Given the sample data
$(x^{(\ell)})_{1\le \ell \le n}$, the sequence $\mathcal{I}$ \eqref{index-subset} defines a mini-batch of batch-size $B$:
\begin{equation}
  \mathcal{B} := (x^{(\ell_1)}, x^{(\ell_2)}, \dots, x^{(\ell_B)})\,,
  \label{eq-mini-batch}
\end{equation}
with which we can approximate the mean $\mathbf{E}_\mu(f)$ by
\begin{equation}
  \EE(f) :=
  \frac{\sum_{i=1}^B f(x^{(\ell_i)})\,\upsilon_{\ell_i}}{\sum_{i=1}^B \upsilon_{\ell_i}}\,.
    \label{mean-wts-mu-empirical-batch}
\end{equation}

In the following we apply \eqref{mean-wts-mu-empirical-batch} to approximating the integrals that will be involved in our algorithm. 
Recall that the Rayleigh quotient is defined as 
\begin{equation}
  \RQ(f) := \frac{\mathcal{E}(f)}{\Var(f)}\,, \quad \mbox{for}~ f \in L^2(\mu)\,,
  \label{def-rayleigh-quotient}
\end{equation}
where $\mathcal{E}(\cdot)$ is the energy \eqref{energy} and $\Var(\cdot)$ is
the variance in \eqref{def-var-cov}. Given the mini-batch
\eqref{eq-mini-batch}, we can approximate the quantities in both \eqref{def-var-cov} and \eqref{def-rayleigh-quotient} by
\begin{equation}
  \begin{aligned}
    \ECov(f,g) :=& \EE\big(fg \big) - \EE(f)\EE(g) , \\
    \EVar(f) :=& \EE\big(f^2\big) - \big(\EE(f)\big)^2 , \\
    \ERQ(f) :=& \frac{\frac{1}{\beta}\EE((a\nabla f)\cdot \nabla f)}{\EVar(f)}\,, 
  \end{aligned}
  \label{E-cov-var-rq}
\end{equation}
respectively, for functions $f,g: \mathbb{R}^d\rightarrow \mathbb{R}$. 

With the above preparations, we are ready to present the learning task for computing the first $K$ eigenpairs of $-\mathcal{L}$.
\paragraph{Learning task for the first $K$ eigenpairs $((\lambda_i, \varphi_i))_{1 \le i \le K}$.} \mbox{}

Let $\mathcal{S}$ be the neural network space \eqref{nn-space} with the neural
network architecture $\mathcal{N} = (N_0, N_1, \dots, N_L)$ in
\eqref{def-arch-nn}, where $L, N_0, N_1, \dots, N_L\in \mathbb{N}$. We assume that $N_0=d$ and $N_L=1$, since 
we want to approximate eigenfunctions which are from $\mathbb{R}^d$ to $\mathbb{R}$.
Given a $C^1$-smooth activation function $\rho:\mathbb{R}\rightarrow
\mathbb{R}$, recall that $\RmapNN:\mathcal{S}\rightarrow C(\mathbb{R}^d,
\mathbb{R})$ is the realization map defined in \eqref{x-to-y}--\eqref{realization-map}.
Let $\alpha>0$ be a positive (penalty) constant and $(\omega_i)_{1\le i \le
K}$ be a decreasing sequence of positive numbers, i.e.,\ $\omega_1 > \dots > \omega_K >0$ (see the second item of Remark~\ref{rmk-algo-choice-of-k-alpha}).

We approximate the $K$ eigenfunctions $(\varphi_i)_{1\le i \le K}$ by the
realizations $(\RmapNN(\Phi_i))_{1\le i \le K}$ of $K$ neural networks
$(\Phi_i)_{1\le i \le K}\subset \mathcal{S}$, which are trained using the loss function 
\begin{align}
  \begin{split}
    \mathrm{Loss}\big(\Phi_1, \dots, \Phi_K\,;&\, \mathcal{I}\big) :=
    \sum_{i=1}^K\omega_i \ERQ(\RmapNN(\Phi_i)) \\
    & + \alpha 
     \sum_{1\le i \le j \le K} \Big(\ECov(\RmapNN(\Phi_{i}), \RmapNN(\Phi_{j})) - \delta_{ij} \Big)^2 \,,
  \end{split}
      \label{lambda-kth-variational-alpha}
\end{align}
where $\mathcal{I}$ is a sequence of indices generated randomly (see
\eqref{index-subset}), $\ERQ(\cdot)$ and $\ECov(\cdot,\cdot)$ are the quantities defined in \eqref{E-cov-var-rq} using the min-batch \eqref{eq-mini-batch}.  In other words, we define the loss function \eqref{lambda-kth-variational-alpha}
based on the variational formulation \eqref{variational-all-first-k}--\eqref{f-orthonormal} in
  Theorem~\ref{thm-variational-form}, where the constraints in \eqref{f-orthonormal} are imposed by adding quadratic penalty terms in
\eqref{lambda-kth-variational-alpha} (see \eqref{constraints-by-cov}).
In particular, when $K=1$, we obtain the learning task for the first eigenpair $(\lambda_1, \varphi_1)$ with the loss 
\begin{align}
  \begin{split}
    \mathrm{Loss}\big(\Phi\,;\,\mathcal{I}\big) :=& \ERQ(\RmapNN(\Phi)) + \alpha
    \Big( \EVar\big(\RmapNN(\Phi)\big) - 1\Big)^2\,.
  \end{split}
  \label{lambda1-u1-variational-alpha}
\end{align}

Denote by $(\Phi^{(j)}_i)_{1\le i\le K}$ and $\mathcal{I}^{(j)}$ the neural networks 
and the sequence of indices \eqref{index-subset} in $j$th training step, respectively, where $j\ge 0$. The first $K$ eigenpairs can be estimated by
\begin{equation}
  \begin{aligned}
    \lambda_i^{(j)} :=& \ERQJ(\RmapNN(\Phi^{(j)}_i)) \,,\\
    \varphi_i^{(j)} :=& \RmapNN(\Phi^{(j)}_i) - \EEJ(\RmapNN(\Phi^{(j)}_i))\,,
  \end{aligned}
  \label{alg1-eigenpairs}
\end{equation}
for $i \in \{1,2,\dots, K\}$.
The complete algorithm for training the neural networks is summarized in Algorithm~\ref{algo-1}.

\begin{algorithm}[t]
  \caption{Compute the first $K$ eigenpairs $((\lambda_i, \varphi_i))_{1 \le i \le K}$.}
  \label{algo-1}
  \begin{algorithmic}[1]
    \State \textbf{Data}: $(x^{(\ell)})_{1\le \ell \le n}$ and their weights $(\upsilon_\ell)_{1\le \ell \le n}$ (see \eqref{mean-wts-mu-empirical}). 
    \State
    \textbf{Parameters}: decreasing sequence $(\omega_i)_{1\le i \le K}$, neural network space $\mathcal{S}$, total number of training steps $J$, penalty parameter $\alpha$, learning rate $r$, batch-size $B$.
    \State
    \textbf{Initialization}: $\Phi^{(0)}_1, \Phi^{(0)}_2, \dots, \Phi^{(0)}_K \in \mathcal{S}$. Set $j=0$.
    \While{$j<J$}
    \State
    Randomly generate the sequence $\mathcal{I}^{(j)}$ of length $B$ as in
    \eqref{index-subset}, and let $\mathcal{B}^{(j)}$ be the corresponding mini-batch in \eqref{eq-mini-batch}.
    \State
    Estimate the eigenpairs $((\lambda_i^{(j)}, \varphi_i^{(j)}))_{1\le i\le K}$ by \eqref{alg1-eigenpairs}. \label{step-algo-1-stat}
    \State (optional) Sort $(\Phi^{(j)}_i)_{1\le i \le K}$ such that the
    eigenvalues $\lambda_1^{(j)}, \dots \lambda_K^{(j)}$ are non-decreasing. 
    \State
Evaluate the loss \eqref{lambda-kth-variational-alpha} using the mini-batch $\mathcal{B}^{(j)}$ by auto-differentiation.
    \label{step-algo-1-forward}
    \State
    Compute the gradient of the loss with respect to neural network parameters by auto-differentiation.
    \label{step-algo-1-gradient}
    \State 
    Update $\Phi^{(j)}_1, \Phi^{(j)}_2, \dots, \Phi^{(j)}_K$ according to the
    gradient in Step~\ref{step-algo-1-gradient} to get $\Phi^{(j+1)}_1, \Phi^{(j+1)}_2, \dots, \Phi^{(j+1)}_K$.
  \EndWhile
    \State \textbf{Output}: statistics based on the estimations recorded in Step~\ref{step-algo-1-stat}.
  \end{algorithmic}
\end{algorithm}
We conclude this section with the following remarks. 
\begin{remark}
  Two comments on the above training task are in order.
  \begin{enumerate}
    \item
Note that, instead of the Rayleigh quotient \eqref{def-rayleigh-quotient}, one can also use the energy $\mathcal{E}$ \eqref{energy} in the loss function \eqref{lambda-kth-variational-alpha}. 
    \item
      Because the Rayleigh quotient $\RQ(f)$ \eqref{def-rayleigh-quotient}
      (respectively, the energy $\mathcal{E}(f)$ \eqref{energy}) involves the spatial
      derivatives of the function $f$, the loss function \eqref{lambda-kth-variational-alpha} involves spatial derivatives of the realizations 
$(\RmapNN(\Phi_i))_{1\le i \le K}$ of neural networks.
For this reason, we choose the activation function $\rho$ to be $C^1$-smooth.
      Also, in Step \ref{step-algo-1-forward} of Algorithm~\ref{algo-1}, we
      need to use auto-differentiation to compute the spatial derivatives 
of $(\RmapNN(\Phi_i))_{1\le i \le K}$ in order to evaluate the loss function. 
  \end{enumerate}
  \label{rmk-algo-1}
\end{remark}
\begin{remark}
We discuss the choices of the parameters $K$, $(\omega_i)_{1\le i \le K}$ and $\alpha$. 
  \begin{enumerate}
    \item 
      Concerning the choice of $K$, Algorithm~\ref{algo-1} works in principle for any $K\in \mathbb{N}$.
      In practice, however, one should choose $K$ depending on the concrete
      problems under consideration and also taking the computational cost into account (the computational cost is larger for larger $K$). For many
      metastable molecular systems, the eigenvalue problem has $k$ small
      eigenvalues $0< \lambda_1 \le \lambda_2 \le \dots \le \lambda_k$ for
      some $k\in \mathbb{N}$ and there is a spectral gap between $\lambda_k$
      and $\lambda_{k+1}$. For these applications, it is appropriate to apply Algorithm~\ref{algo-1} with some $K\in \{1,\dots,k\}$. 
    \item
      Although Theorem~\ref{thm-variational-form} holds for non-increasing sequences $\omega_1 \ge \omega_2 \ge \cdots \ge \omega_K>0$, 
      in practice, choosing $(\omega_i)_{1\le i \le K}$ to be pairwise distinct, i.e.,\ $\omega_1 > \omega_2 > \cdots > \omega_K>0$, 
      can help avoid the non-uniqueness of the minimizer due to reordering of eigenfunctions. 
For problems where the true eigenvalues $\lambda_1, \dots, \lambda_K$ are of the same order,  
      Algorithm~\ref{algo-1} works well for different sequences $(\omega_i)_{1\le i \le K}$ as long as the ratio
      $\frac{\omega_K}{\omega_1}$ is not too small (so that each eigenvalue has similar contribution to the total loss in (38)).
Even when the true eigenvalues have different orders of magnitude, Algorithm~\ref{algo-1} works by choosing the parameters $(\omega_i)_{1\le i \le K}$ properly (see the alanine dipeptide example in Section~\ref{subsec-numerical-ad}).
    \item
      A large $\alpha$ is required in Algorithm~\ref{algo-1} in order to guarantee that the constraints are imposed effectively.
      However, a too large $\alpha$ would introduce stiffness which in turn
      restricts the size of the learning rate in training. As an example, when
      the coefficients $(\omega_i)_{1\le i \le K}$ are chosen such that
      $\sum_{i=1}^K \omega_i\lambda_i$ are below $5$ (as $(\lambda_i)_{1\le
      i \le K}$ are unknown, this may require some empirical estimates in practice), then $\alpha \in [20, 50]$ would be an appropriate choice.
  \end{enumerate}
  \label{rmk-algo-choice-of-k-alpha}
\end{remark}

\section{Numerical examples}
\label{sec-example}
In this section, we study two concrete examples in order to demonstrate Algorithm~\ref{algo-1}. 
The code used to produce the numerical results in this section is available at \url{https://github.com/zwpku/EigenPDE-NN}.

\subsection{A high-dimensional model problem}
\label{subsec-numerical-model-problem}

In the first example, we consider an eigenvalues problem whose leading eigenvalues can be
computed using traditional numerical methods. This example allows us to compare the
solutions given by Algorithm~\ref{algo-1} to the solutions computed by traditional numerical methods (i.e.,\ reference solutions). 
We consider the problem~\eqref{eigen-problem-rd} for different dimensions
$d=2,50,100$. In each case, we chose $\beta=1.0$ and we fix the matrix $a$ in \eqref{generator-l} as the identity matrix. Correspondingly, the generator \eqref{generator-l} is 
  \begin{align}
    \mathcal{L}_d f =& - \nabla V_d \cdot \nabla f + \Delta f 
    \label{generator-l-ex1}
  \end{align}
for a test function $f: \mathbb{R}^d\rightarrow \mathbb{R}$, where $V_d: \mathbb{R}^d\rightarrow \mathbb{R}$ for $d=2,50,100$ are the potential functions that we explain next. 

\paragraph{Potentials $V_d$ for $d=2,50,100$.} First, let us consider the case where $d=2$. The potential $V_2:\mathbb{R}^2\rightarrow \mathbb{R}$ is defined as 
\begin{equation}
  V_2(x_1, x_2) = V(\theta) + 2 (r-1)^2 + 5 {\mathrm{e}}^{-5 r^2}\,, \quad \forall~ (x_1,x_2)\in \mathbb{R}^2\,,
  \label{pot-V2}
\end{equation}
where $(\theta,r) \in [-\pi, \pi)\times [0, +\infty)$ are the polar
coordinates which are related to $(x_1, x_2)\in \mathbb{R}^2$ by
\begin{equation}
  x_1 = r\cos\theta, \quad  x_2 = r\sin\theta \,,
\end{equation}
and $V: [-\pi,\pi) \rightarrow \mathbb{R}$ is a double-well potential function defined as 
\begin{equation}
 \quad 
  V(\theta)=
      \begin{cases} 
	  \big[1-(\frac{3\theta}{\pi} + 1)^2\big]^2\,, & \theta \in[-\pi, -\frac{\pi}{3})\,,\\
	\frac{1}{5}(3 - 2 \cos(3\theta))\,, &  \theta \in[-\frac{\pi}{3}, \frac{\pi}{3})\,,\\
	\big[1-(\frac{3\theta}{\pi} - 1)^2\big]^2\,,& \theta \in [\frac{\pi}{3}, \pi)\,.
      \end{cases}
      \label{ex1-pot-v0}
\end{equation}
As shown in the right plot of Figure~\ref{fig-ex1-pot2d}, there are three
low-energy regions on the potential surface of $V_2$, which are labelled as $A$, $B$, and $C$. Each of the two regions
$A$ and $B$ contains a global minimum point of $V_2$ (i.e.,\ both of these two
minimum points attain the same lowest potential value), while the region $C$ contains a local minimum point of $V_2$. 

Next, for $d=50$ and $d=100$, we define $V_d: \mathbb{R}^d\rightarrow
\mathbb{R}$ as the sum of $V_2$ in the first two coordinates of the state and a Gaussian potential in the other coordinates, namely,
\begin{equation}
  V_d(x) = V_2(x_1,x_2) + 5\sum_{i=3}^{d} x_i^2\,,\quad \forall~ x=(x_1,x_2, \dots, x_d) \in \mathbb{R}^d\,.
  \label{pot-vd}
\end{equation}
Roughly speaking, the coefficient $5$ in front of the Gaussian term in
\eqref{pot-vd} is introduced such that the dynamics of the coordinates $(x_3, \dots,
x_d)$ under the potential $V_d$ reaches quasi-equilibrium in a sufficiently
short time. This in turn guarantees that the three smallest eigenvalues
$\lambda_1$, $\lambda_2$, and $\lambda_3$ of $-\mathcal{L}_d$, where $d=50,
100$, are the same as those when $d=2$, so that we can
use the solution given by the finite volume method for $d=2$ (see \eqref{ref-eigval-solutions} below) as the reference solution 
to the eigenvalue PDE \eqref{eigen-problem-rd} for dimensions $d=50,100$ as well.
Correspondingly, for both $d=50$ and $d=100$, the first three eigenfunctions
$\varphi_1$, $\varphi_2$, $\varphi_3$  are functions of the first two
coordinates $(x_1, x_2)$ only, and can be compared to the eigenfunctions
obtained using the finite volume method for $d=2$.

\begin{figure}[t!]
  \centering
  \includegraphics[width=\textwidth]{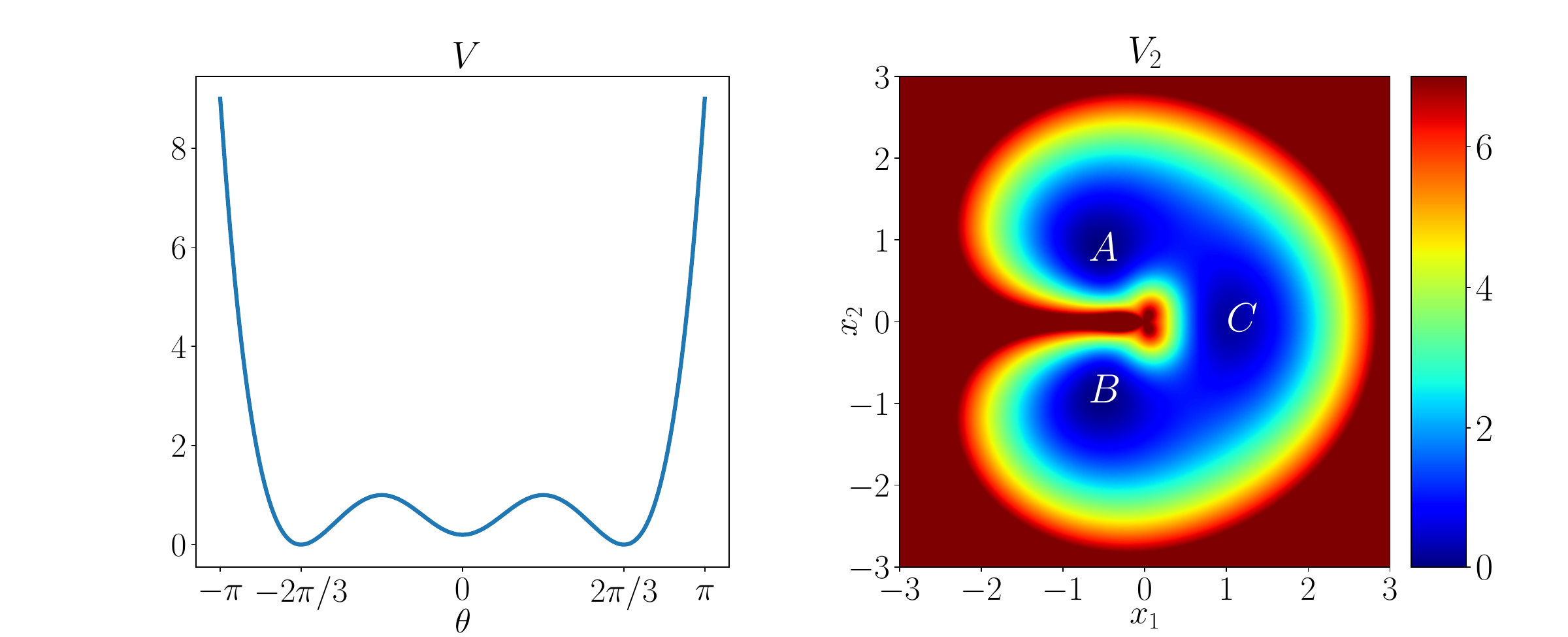}
  \caption{Profiles of the potential $V$ \eqref{ex1-pot-v0} (left) and the two-dimensional
  potential $V_2$ \eqref{pot-V2} (right) in the first example. There are three low-energy regions
  on the potential surface of $V_2$, which are labelled as $A$, $B$, and $C$.
  Regions $A$ and $B$ contain global minimum points of $V_2$, whereas the region $C$ contains a local minimum point of $V_2$.
  \label{fig-ex1-pot2d}}
\end{figure}

\begin{figure}[t!]
  \centering
  \includegraphics[width=\textwidth]{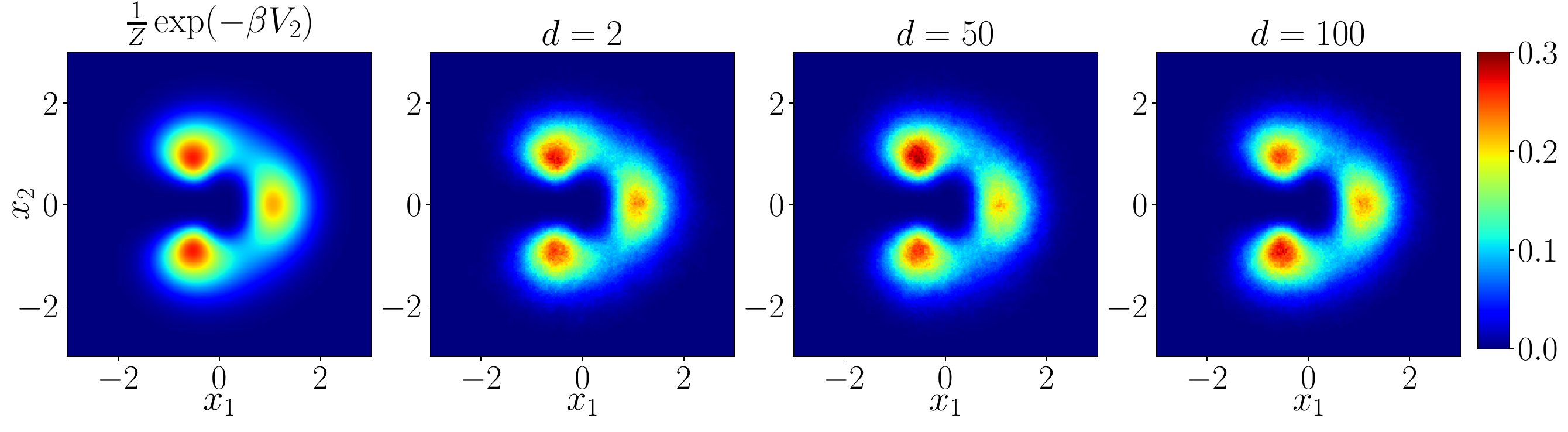}
  \caption{Empirical densities in the first example. First plot: the
  probability density $\frac{1}{Z}\mathrm{e}^{-\beta V_2}$, where $\beta=1$
  and $V_2$ is the potential function defined in \eqref{pot-V2} (see its
  profile in Figure~\ref{fig-ex1-pot2d}). Second plot: empirical probability
  density of the sample data for $d=2$. Third and fourth plots: empirical marginal probability densities 
  of the sample data in the first two coordinates $(x_1, x_2)$ for $d=50$ and $d=100$, respectively.
  In each of the last three plots, $n=5\times 10^6$ states are sampled using
  the Euler-Maruyama scheme with the timestep $\Delta t=0.001$.  
  \label{fig-ex1-hist}}
\end{figure}

\begin{table}[h]
  \centering
  \begin{tabular}{c|c|c|c|c}
    \hline
    \hline
    & FVM, $d=2$ & NN, $d=2$ & NN, $d=50$ & NN, $d=100$ \\
    \hline
    $\lambda_1$ & $0.219$ & $0.226$ $(0.003)$ & $0.210$ $(0.002)$ & $0.220$ $(0.003)$\\
    \hline
  $\lambda_2$ & $0.764$  & $0.768$ $(0.006)$ & $0.751$ $(0.007)$& $0.749$ $(0.006)$\\ 
    \hline
$\lambda_3$ & $2.790$ & $2.810$ $(0.017)$ &$2.725$ $(0.022)$ &$2.737$ $(0.020)$\\
    \hline
  \end{tabular}
  \caption{Estimations of the eigenvalues $\lambda_1, \lambda_2, \lambda_3$ in the first example. 
 Column ``FVM, $d=2$'' shows the numerical solutions obtained using the finite
  volume method for $d=2$ (see~\eqref{ref-eigval-solutions}). 
  Columns with labels ``NN, $d=2$'', ``NN, $d=50$'', and ``NN, $d=100$'' show
  the means and the sample standard deviations (in brackets) of the
  eigenvalues estimated in the last $100$ training steps of
  Algorithm~\ref{algo-1} for $d=2$, $d=50$, and $d=100$, respectively.  The
  potential function is $V_2$ \eqref{pot-V2} for $d=2$, whereas the potential
  functions $V_d$ for $d=50$ and $d=100$ are defined in \eqref{pot-vd}.
  \label{tab-ex1}}
\end{table}

\begin{figure}[h!]
  \centering
  \includegraphics[width=13cm,height=9cm]{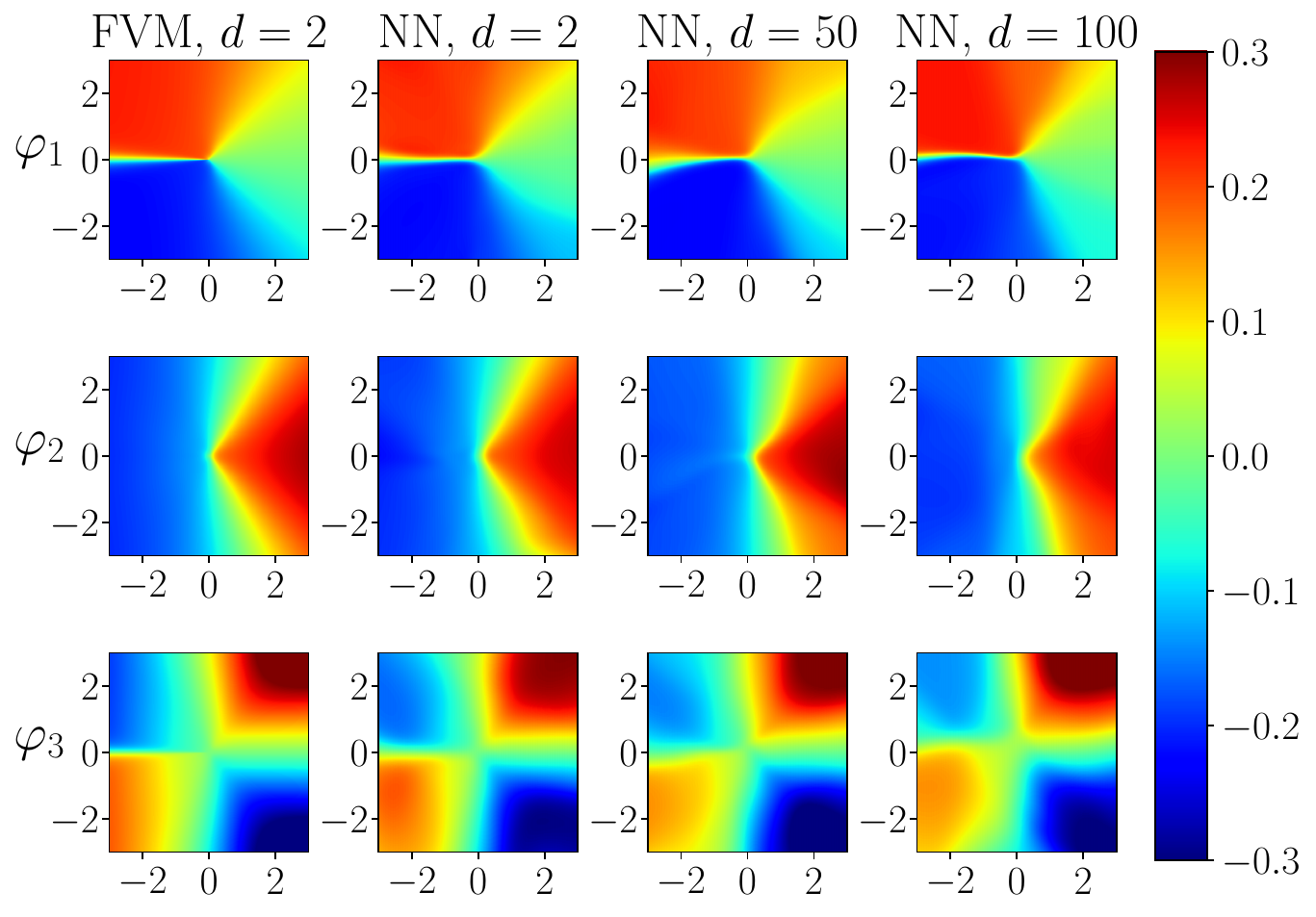}
  \caption{The first three eigenfunctions $\varphi_1, \varphi_2, \varphi_3$ in
  the first example, computed using the finite volume method for $d=2$ (column ``FVM,
  $d=2$'') and by training neural networks using Algorithm~\ref{algo-1} for
  $d=2$ (column ``NN, $d=2$''), $d=50$ (column ``NN, $d=50$''), and $d=100$
  (column ``NN, $d=100$''), respectively. The last three columns show the
  trained eigenfunctions after $J=7100$ training steps using
  Algorithm~\ref{algo-1}.  For $d=50$ and $d=100$, the third and the fourth
  columns show the eigenfunctions $\varphi_1, \varphi_2, \varphi_3$ as
  functions in the first two coordinates $x_1, x_2$, where the remaining coordinates $(x_3, \dots, x_d)$ are randomly selected 
according to certain centered Gaussian distribution. \label{fig-ex1-evs-all}}
\end{figure}

\begin{figure}[h!]
  \centering
  \includegraphics[width=12cm]{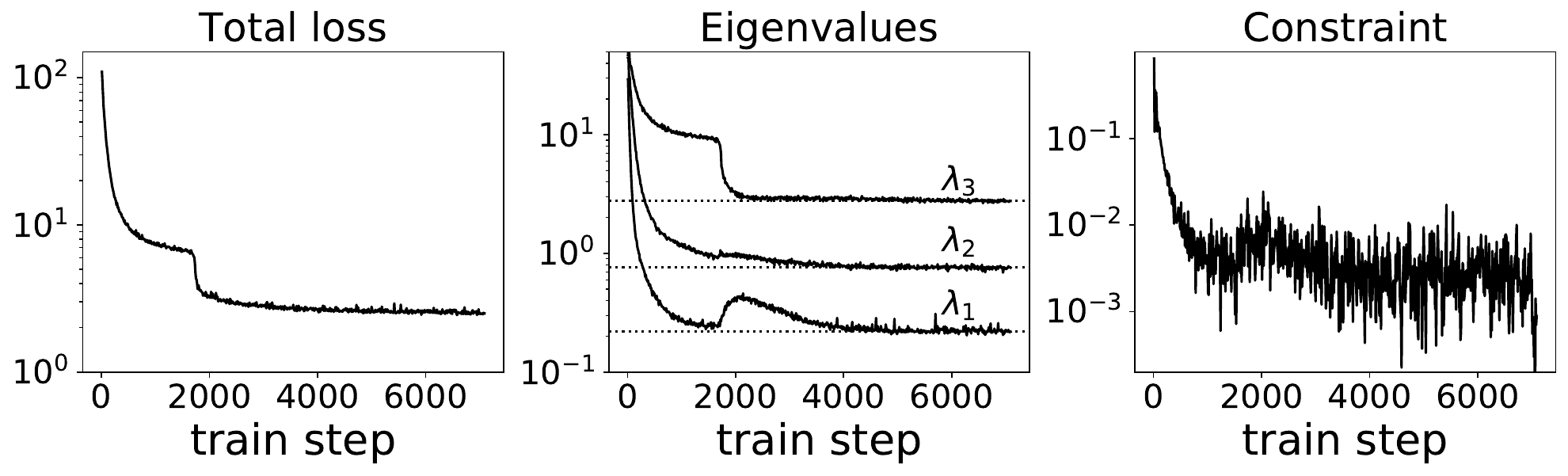}
  \caption{The evolution of quantities during the training procedure in the first example, where $d=100$.
  Left: the loss function \eqref{lambda-kth-variational-alpha}. Middle: estimations of the eigenvalues $\lambda_1$, $\lambda_2$, and $\lambda_3$. 
  The dotted horizontal lines show the reference solutions \eqref{ref-eigval-solutions} obtained using the finite volume method.
  Right: the quantity $\mathcal{C}$ in \eqref{ex1-penalty-c1} which corresponds to the penalty term in \eqref{lambda-kth-variational-alpha}. 
  \label{fig-ex1-loss_constraints}}
\end{figure}

\paragraph{Reference solution for $d=2$.}
Since \eqref{eigen-problem-rd} is a two-dimensional eigenvalue PDE problem
when $d=2$, the eigenvalues of $-\mathcal{L}_2$, given by
\eqref{generator-l-ex1}, can be solved numerically using the finite volume method~\cite{latorre2011}. Specifically, taking into account the profile of the potential surface of $V_2$ (see the right plot
of Figure~\ref{fig-ex1-pot2d}), we truncate the space $\mathbb{R}^2$ to the
finite domain $[-3.0, 3.0]\times [-3.0, 3.0]$, on which the PDE
\eqref{eigen-problem-rd} is discretized using a grid of size $400\times
400$ (see \cite{sz-entropy-2017} for details of the discretization). The
discretized matrix eigenvalue problem is then solved using Krylovschur method
implemented in \textbf{slepc4py}, which is a Python binding for the package
SLEPc~\cite{slepc}. In this way, we obtain the first three eigenvalues as
\begin{equation}
\lambda_1=0.219, \quad \lambda_2=0.764,\quad \lambda_3 =2.790\,. 
  \label{ref-eigval-solutions}
\end{equation}
These values in \eqref{ref-eigval-solutions} remain unchanged when we enlarge the truncated domain and when we refine the grid used for discretization.
The corresponding eigenfunctions $\varphi_1, \varphi_2, \varphi_3$ obtained
using the finite volume method are shown in the first column of Figure~\ref{fig-ex1-evs-all}.

\paragraph{Solutions for $d=2$, $50$, and $100$ using neural networks.} Next, we use Algorithms~\ref{algo-1} to compute the first three eigenpairs of $-\mathcal{L}_d$ \eqref{generator-l-ex1} for dimensions $d=2$, $d=50$ and $d=100$, respectively. 

For each $d\in \{2,50,100\}$, the invariant measure $\mu$ corresponding to the generator $\mathcal{L}_d$  has the density $\frac{1}{Z}
\mathrm{e}^{-\beta V_d}$, where $Z$ is the normalizing constant (depending on $d$). 
We first generate $n=5\times 10^6$ states in $\mathbb{R}^d$ from the scheme
\begin{equation}
  x^{(\ell)} = x^{(\ell-1)} - \nabla V_d(x^{(\ell-1)})\,\Delta t +
  \sqrt{2\beta^{-1}\Delta t} \bm{\eta}^{(\ell)}\,, \quad \ell = 1, 2, \dots, n\,,
  \label{direct-sde-discretization}
\end{equation}
starting from some initial state $x^{(0)}\in \mathbb{R}^d$, where the timestep
$\Delta t=0.001$ and $\bm{\eta}^{(\ell)}\in \mathbb{R}^d$, $1 \le \ell \le n$,
are i.i.d.\ standard Gaussian variables in $\mathbb{R}^d$. Note that
\eqref{direct-sde-discretization} is just the Euler-Maruyama scheme of the
SDE~\eqref{overdamped} when both $a$ and $\sigma$ are the identify matrix.
In this example we do not use the reweighting technique introduced in
Section~\ref{sec-algo-consideration}, since it is not difficult to sample states directly from the numerical scheme \eqref{direct-sde-discretization} whose invariant measure approximates $\mu$. 
In other words, we have $\upsilon_{\ell}=1$ for $1 \le \ell \le n$, where $\upsilon_{\ell}$ are the weights in \eqref{mean-wts-mu-empirical}.
As shown in Figure~\ref{fig-ex1-hist}, the empirical (marginal) probability densities of the sampled trajectory data
in $(x_1, x_2)$ are accordant with the probability density $\frac{1}{Z} \mathrm{e}^{-\beta V_2}$. 
This implies that the trajectories are sufficiently long, so that the
empirical estimator \eqref{mean-wts-mu-empirical}, as well as the estimator
\eqref{mean-wts-mu-empirical-batch} for large batch-sizes, are good
approximations of the true mean value $\mathbf{E}_\mu(f)$ (for reasonable functions $f$).

Using the sampled states as training data, we apply Algorithm~\ref{algo-1} to estimate the first three eigenpairs.
 We set $K=3$, the penalty parameter $\alpha=20$, and the coefficients $\omega_1=1.0$, $\omega_2=0.8$, and $\omega_3=0.6$ in the loss function \eqref{lambda-kth-variational-alpha}.
For each $d\in\{2,50,100\}$, each of the first three eigenfunctions is represented
by a neural network with the same network architecture 
\begin{equation}
\mathcal{N}=(d, 20, 20, 20,1)\,. 
  \label{ex1-arch}
\end{equation} 
In other words, the neural network has one input layer of size $d$, three
hidden layers of size $20$, and one output layer of size $1$ (see
Figure~\ref{fig-nn} for the illustration of neural networks). We use the activation function $\rho(x) = \tanh x$. 
In each test, in order to train the neural network, $J=7100$ training steps
are performed using the Adam optimization method~\cite{adam-KingmaB14} with learning rate $r= 0.005$.
The batch-size $B=5000$ is used for the first $7000$ steps.  The eigenvalues
are computed (see Table~\ref{tab-ex1}) as the mean values of the estimations
in the final $100$ training steps, i.e.,\ from step $7001$ to step $7100$,
where a large batch-size $B=20000$ is adopted.
As seen from Table~\ref{tab-ex1} and Figure~\ref{fig-ex1-evs-all}, Algorithm~\ref{algo-1} is able to approximate the first three eigenvalues in \eqref{ref-eigval-solutions} and their corresponding eigenfunctions. 
For $d=50, 100$, by inspecting the eigenfunctions at coordinates $(x_3, \dots,
x_d)$ that are randomly sampled according to certain centered Gaussian distribution, we find that they depend on the values of
$(x_3, \dots, x_d)$ rather weakly (see the last two columns of Figure~\ref{fig-ex1-evs-all}).
  Although the potentials \eqref{pot-vd} in this example are relatively simple, it is interesting to note that, by training neural
  networks with fully connected architecture \eqref{ex1-arch},
  Algorithms~\ref{algo-1} is able to identify the eigenfunctions which
  are functions of $(x_1, x_2)$ only. Figure~\ref{fig-ex1-loss_constraints}
  shows the evolution of the loss function
  \eqref{lambda-kth-variational-alpha}, the estimations of eigenvalues using
  \eqref{alg1-eigenpairs} (see Step~\ref{step-algo-1-stat} of Algorithm~\ref{algo-1}), and the quantity 
  \begin{equation}
      \mathcal{C} = \sum_{1\le i \le j \le K} 
    \Big(\ECov\big(\RmapNN(\Phi_{i}),\RmapNN(\Phi_{j}) \big) - \delta_{ij}\Big)^2 
    \label{ex1-penalty-c1}
  \end{equation}
during the training procedure in the case where $d=100$. The results for
$d=2,50$ are similar and therefore they are not shown here.
The quantity $\mathcal{C}$ in \eqref{ex1-penalty-c1} corresponds to the
penalty term in the loss function \eqref{lambda-kth-variational-alpha}.
One can observe from Figure~\ref{fig-ex1-loss_constraints} that both the loss function and the eigenvalues converge within $5000$ training steps.  The quantity $\mathcal{C}$ \eqref{ex1-penalty-c1} is fluctuating during the
training procedure due to the use of both mini-batch and the finite penalty parameter $\alpha=20$. Nevertheless, the magnitude of $\mathcal{C}$ stays
below $10^{-2}$ in most of the training steps, indicating that the constraints are well imposed during the training procedure.
Finally, we note that very similar results were obtained when we carried out the same numerical experiment with a larger network architecture $\mathcal{N}=(d, 25, 25,25,25, 1)$. 

\subsection{Alanine dipeptide}
\label{subsec-numerical-ad}
In the second example, we study the simple molecular system alanine dipeptide in vacuum.  
The system consists of $22$ atoms. Since each atom has three coordinates, the
full state of the system has dimension $66$. It is known that the dynamics of the system can be well described using two dihedral angles $\phi_1, \phi_2$ (see Figure~\ref{fig-alanine-illustration}). The system exhibits three metastable conformations, which are often named as
C5, C7eq and C7ax (see Figure~\ref{fig-ad-free-energy}). The transition
between the two conformations C7eq and C7ax occurs much more rarely comparing to the transition between the conformations C5 and C7eq.

\begin{figure}[t!]
  \centering
  \includegraphics[width=0.8\textwidth]{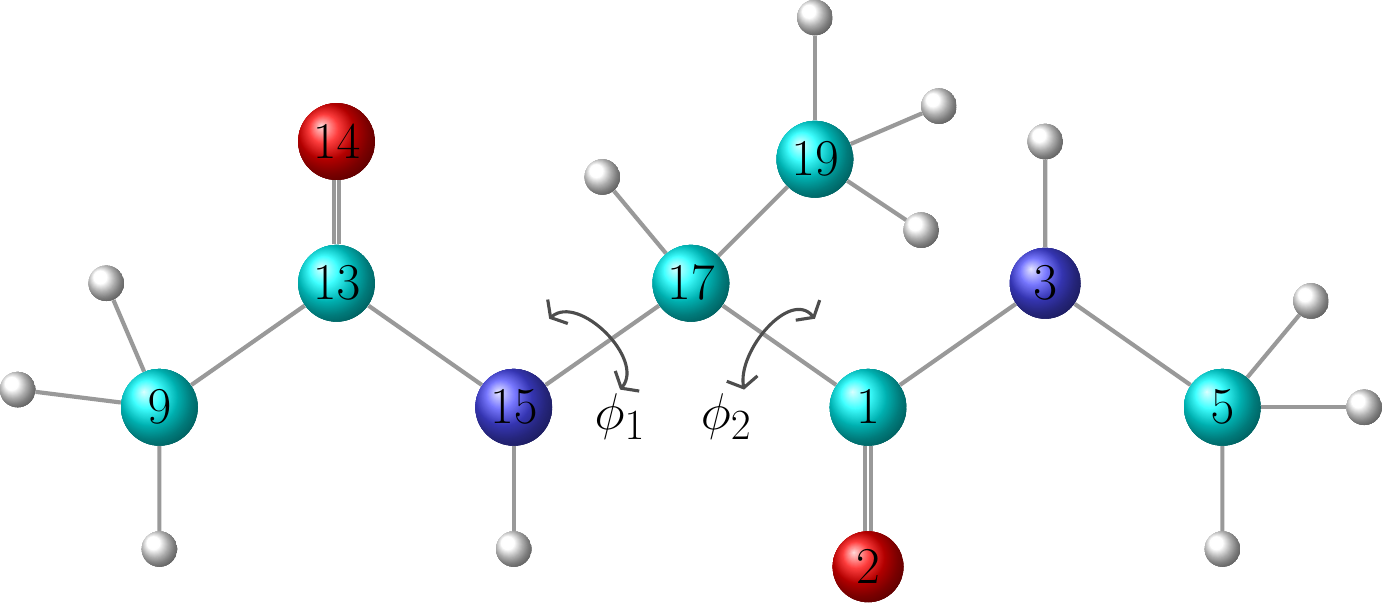}
  \caption{Alanine dipeptide in vacuum. The $22$ atoms of the system as well as the two dihedral angles $\phi_1, \phi_2$ are shown.
  The hydrogen atoms, carbon atoms, nitrogen atoms and oxygen atoms are
  displayed in gray, cyan, blue, and red, respectively. The indices of the $10$ non-hydrogen atoms are given by the numbers. 
  \label{fig-alanine-illustration}}
\end{figure}

We generate the trajectory data of the system using the NAMD software
package~\cite{namd-paper}. In all the simulations below, the system is simulated using Langevin dynamics at the temperature $T=\SI{300}{\kelvin}$ with the damping coefficient \SI{1}{\per\s} and the timestep \SI{1}{\femto\s}.
The dynamics of the system in the position space is ergodic
with respect to the unique invariant distribution $\mu$ \eqref{mu-invariant}
for some potential function $V:\mathbb{R}^{66}\rightarrow \mathbb{R}$, where
$\beta=(k_BT)^{-1}=1.678 (\si{\kilo cal \per \mole})^{-1}$ and $k_B$ denotes
the Boltzmann constant. The initial state of the system is prepared by performing $500$ energy
minimization steps followed by $10^6$ equilibration steps (i.e.,\ \SI{1}{\nano\s}).
Due to the metastability of the system, unbiased molecular dynamics simulation is computationally expensive
for generating trajectory data that is distributed according to the
invariant distribution $\mu$. Therefore, we use the reweighting technique discussed in Section~\ref{sec-algo-consideration} 
and we sample the data from a biased simulation.
Specifically, the training data and the test data are prepared in the following three steps.
\begin{itemize}
  \item[(1)] \textit{Computation of mean force and its potential using ABF.}
    In the first step, we simulate the system for \SI{20}{\nano\s} using
    the adaptive biasing force (ABF) method~\cite{free-energy-by-average-force,overcome-barrier-using-unconstrained-md,abf-everything}
    that is implemented in the \textbf{colvar} module of the NAMD
    package~\cite{colvar}. The two dihedral angles $\phi_1, \phi_2$ are used
    as collective variables in the ABF method, whose space
    $[\SI{-180}{\degree}, \SI{180}{\degree}) \times [\SI{-180}{\degree},
    \SI{180}{\degree})$ is discretized with grid size \SI{5}{\degree}. During the
    simulation, the mean force in each cell of the discretized grid of the
    dihedral angles is estimated using samples that fall in the cell, and is 
    applied to the system (when the system visits the cell) after $100$
    samples are collected. After the simulation, we obtain the mean force and
    its potential $V_{\mathrm{PMF}}$, i.e.,\ the potential of mean force
    (PMF), on the discrete grid of the dihedral angles (see Figure~\ref{fig-ad-free-energy}).
  \item[(2)] 
    \textit{Biasing force by rescaling the mean force}. As one can see in Figure~\ref{fig-ad-free-energy}, the magnitude of the potential $V_{\mathrm{PMF}}$ is quite large. 
    In this step, we rescale the mean force obtained in the previous step by $\eta=0.7$. 
    Accordingly, the potential of the rescaled mean force is $V_{\mathrm{bias}}=0.7V_{\mathrm{PMF}}$.
\item[(3)]
  \textit{Training data and test data by molecular dynamics simulations under fixed biasing force}.
We simulate the system for \SI{100}{\nano\s} using ABF, where the rescaled
    mean force in the previous step is applied to the system and is kept
    fixed during the simulation. The trajectory is recorded every
    \SI{1}{\pico\s}, resulting in a training data set $(x^{(\ell)})_{1\le
    \ell\le n}$ that consists of $n=10^5$ states.
    We denote by $\phi_1^{(\ell)}, \phi_2^{(\ell)}$ the two dihedral angles
    $\phi_1, \phi_2$ of the state $x^{(\ell)}\in \mathbb{R}^{66}$ for $\ell=1,2,\dots, n$.
    Then, the weights 
    \begin{equation}
    \upsilon_{\ell} = \frac{\exp(-\beta V_{\mathrm{bias}}(\phi_1^{(\ell)}, \phi_2^{(\ell)}))}{
      \frac{1}{n}\sum_{\ell'=1}^n \exp(-\beta V_{\mathrm{bias}}(\phi_1^{(\ell')},
      \phi_2^{(\ell')}))}, \quad \ell \in \{1,2,\dots, n\}
    \label{weights-v-ad}
    \end{equation}
    are used in estimating the mean values according to
    \eqref{mean-wts-mu-empirical}. See Figure~\ref{fig-ad-histograms} for the
    histogram of the angles $(\phi_1^{(\ell)}, \phi_2^{(\ell)})_{1\le \ell \le
    n}$ of the trajectory data and the profile of the weights as a function of
    the dihedral angles. Finally, we obtain the test data set of the same size by running another biased simulation independently with the same parameters.
\end{itemize}
\begin{figure}[t!]
  \centering
  \includegraphics[width=0.80\textwidth]{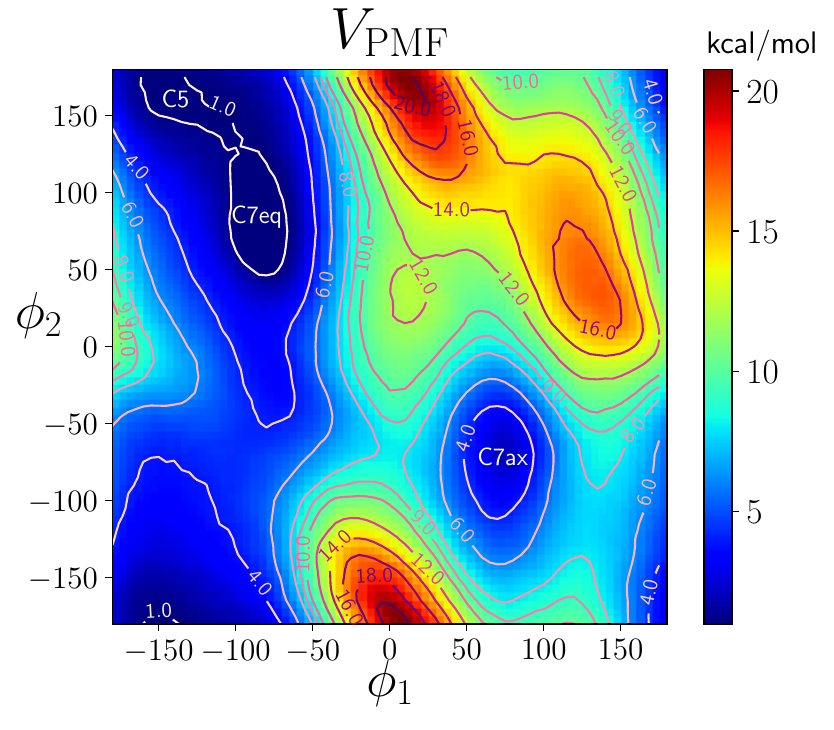}
  \caption{Potential of mean force $V_{\mathrm{PMF}}$ as a function of the two dihedral angles
  $\phi_1, \phi_2$ of alanine dipeptide, computed using the adaptive biasing method in the NAMD package. 
 The system has three metastable conformations, which are termed as C5, C7eq and
  C7ex, respectively. These three conformations correspond to the three regions where the value of the potential $V_{\mathrm{PMF}}$ is low.
  \label{fig-ad-free-energy}}
\end{figure}

Let us point out that, alternative to the ABF method, sampling techniques such
as Metadynamics~\cite{Laio_2008,Laio12562}, the extended system ABF (eABF)
method~\cite{EABF} can be used in preparing data as well. It is also possible
to employ sampling methods that do not require the knowledge of collective
variables, such as the simulated tempering~\cite{Marinari_1992} and the replica exchange molecular dynamics~\cite{RHEE2003775}.
However, in this case, the estimator~\eqref{mean-wts-mu-empirical}
has to be modified in order to estimate the mean values in the loss function.

\begin{figure}[t!]
  \centering
  \includegraphics[width=0.49\textwidth]{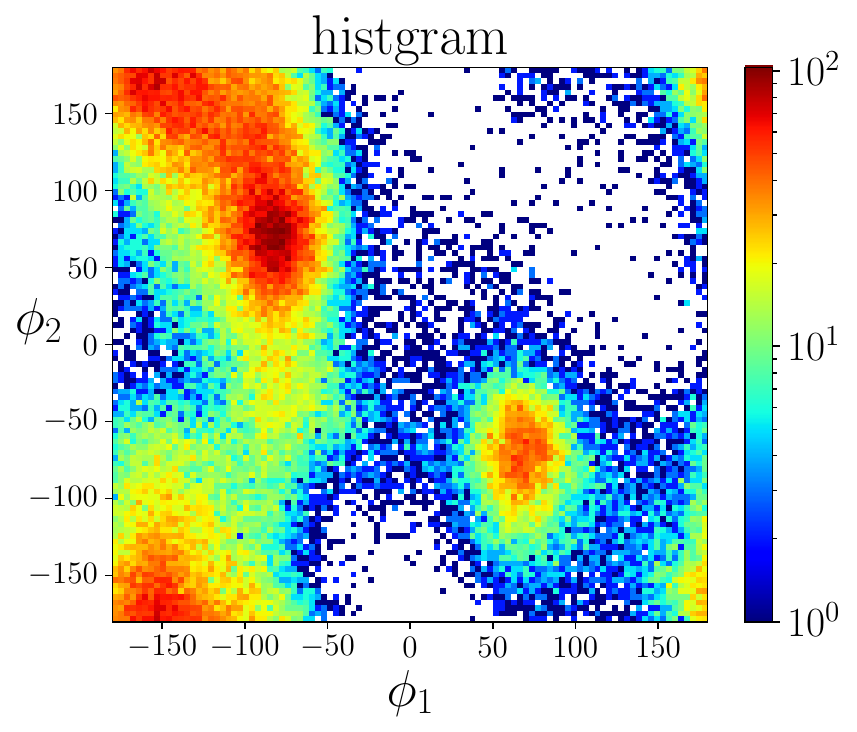}
  \includegraphics[width=0.47\textwidth]{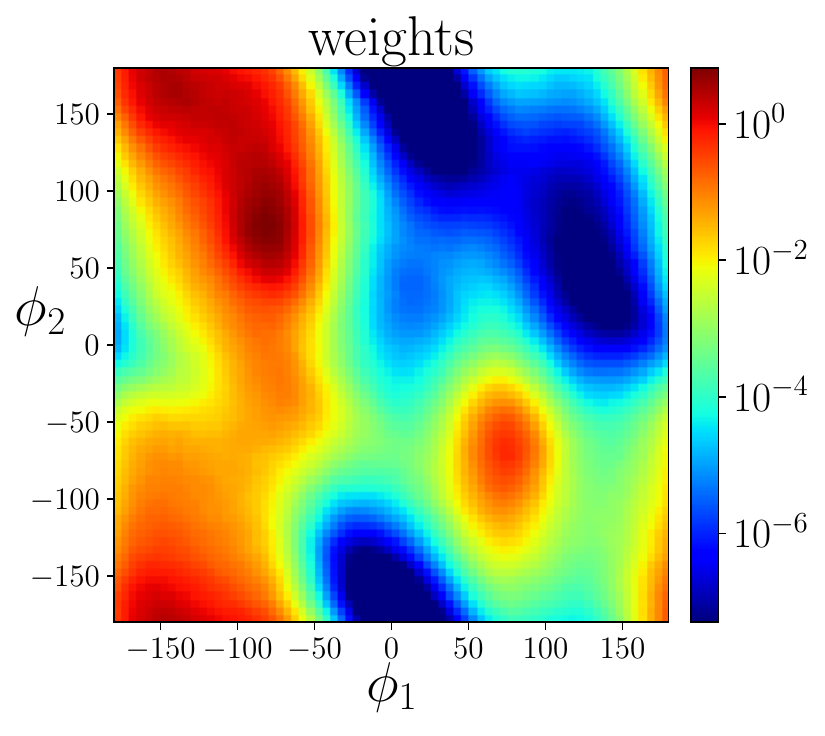}
  \caption{Left: histogram of the two dihedral angles $\phi_1$ and $\phi_2$ of
  alanine dipeptide. The system is simulated for \SI{100}{\nano\s}, under the fixed biasing force whose corresponding potential is $V_{\mathrm{bias}}=
  0.7V_{\mathrm{PMF}}$ (see Figure~\ref{fig-ad-free-energy} for the profile of
  $V_{\mathrm{PMF}}$). By recording the states every \SI{1}{\pico\s}, $10^5$ states of
  the trajectory are obtained in total, which are used to plot the histogram.
  Right: weights proportional to $\exp(-\beta V_{\mathrm{bias}})$ as a function of the dihedral angles (see \eqref{weights-v-ad}).
  \label{fig-ad-histograms}}
\end{figure}

With the training data prepared above, we compute the leading eigenpairs of the problem \eqref{eigen-problem-rd} by applying Algorithm~\ref{algo-1}, where the generator is 
  \begin{align}
    \begin{split}
\mathcal{L}f =& -\frac{D}{k_BT} \nabla V \cdot \nabla f + D \Delta f
    \end{split}
    \label{generator-l-ad}
  \end{align}
  for a test function $f: \mathbb{R}^{66}\rightarrow \mathbb{R}$, and $D>0$ is the diffusion coefficient.
  Equivalently, we are considering the SDE 
  \begin{align}
    dx(s) = - \frac{D}{k_BT}\nabla V(x(s))\,ds + \sqrt{2D} \, dw(s)\,,~ s \ge 0\,,
    \label{overdamped-ad}
  \end{align}
  where $x(s)\in \mathbb{R}^{66}$ and $(w(s))_{s \ge 0}$ is a standard Brownian motion in $\mathbb{R}^{66}$.
Without loss of generality, we assume that the indices of the coordinates
$x=(x_1, x_2, \dots, x_{66})$ are ordered in a way such that the coordinates
of the non-hydrogen atoms are $((x_{3(i-1)+1}, x_{3(i-1)+2}, x_{3(i-1)+3}))_{1 \le i \le 10}$.
We define $\mathbf{x}_i=(x_{3(i-1)+1}, x_{3(i-1)+2}, x_{3(i-1)+3})$ for $i\in
\{1,2,\dots, 10\}$ and set $\mathbf{x} = (x_1,x_2, \dots, x_{10})^T$ (note that the ordering here is different from the indices in Figure~\ref{fig-alanine-illustration}).
  In the following numerical tests we choose $D= \SI{e-5}{\cm^2\per\s}$.

  As in the work~\cite{committor-by-nn,chasing-cv-autoencodes-biased-traj}, we
  approximate the eigenfunctions by functions of $\mathbf{x}\in \mathbb{R}^{30}$, i.e., the coordinates of the $10$
  non-hydrogen atoms (see Figure~\ref{fig-alanine-illustration}). To guarantee
  the eigenfunctions after training are invariant under both rotations and
  translations, the Cartesian coordinates $\mathbf{x}$ of the non-hydrogen atoms in the
  trajectory data are aligned with respect to the coordinates of the same
  predetermined reference configuration (such that the root mean squared
  deviation is minimized) before they are passed to the neural networks. Specifically, 
  we define the map $F_{\mathrm{opt}}: \mathbb{R}^{30}\rightarrow \mathbb{R}^{30}$ as
  \begin{equation}
    \begin{aligned}
    F_{\mathrm{opt}}(\mathbf{x}) =& \Big((\mathbf{x}_1 -
      \mathbf{b}_{\mathrm{opt}}(\mathbf{x}))A_{\mathrm{opt}}(\mathbf{x}),
    (\mathbf{x}_2 -
    \mathbf{b}_{\mathrm{opt}}(\mathbf{x}))A_{\mathrm{opt}}(\mathbf{x}),\\
    &\quad \dots, (\mathbf{x}_{10} - \mathbf{b}_{\mathrm{opt}}(\mathbf{x}))A_{\mathrm{opt}}(\mathbf{x})\Big)^T\,, 
    \end{aligned}
    \label{map-F-ad}
  \end{equation}
  where, for given $\mathbf{x}$, $A_{\mathrm{opt}}(\mathbf{x})\in \mathbb{R}^{3\times 3}, \mathbf{b}_{\mathrm{opt}}(\mathbf{x})\in \mathbb{R}^3$ are the optimal rotation matrix and the optimal translation vector,
  respectively, which minimize the root mean squared deviation of $\mathbf{x}$ from the reference configuration.
  In practice, $\mathbf{b}_{\mathrm{opt}}(\mathbf{x})$ is easily determined by
  matching the centers of atoms, whereas $A_{\mathrm{opt}}(\mathbf{x})$ can be
  numerically computed using the Kabsch algorithm~\cite{Kabsch}.

The eigenfunctions are approximated by functions that are of the form
${\RmapNN(\Phi)\circ F_{\mathrm{opt}}(\mathbf{x})}$, where $\Phi$ is a neural network with the network architecture 
\begin{equation}
\mathcal{N}=(30, 20, 20, 20, 1)\,. 
  \label{ex2-arch}
\end{equation} 
In other words, the Cartesian coordinates $\mathbf{x}\in \mathbb{R}^{30}$ of
the non-hydrogen atoms are aligned using the map
$F_{\mathrm{opt}}$~\eqref{map-F-ad}. Then, they are passed to the neural network
which has three hidden layers of equal size $20$ and one output layer of size
$1$. It is clear that the functions represented in the form above are invariant under both rotations and translations.
As in the previous example, we use $\rho(x) = \tanh x$ as the activation function.

We start by computing the first eigenpair $(\lambda_1, \varphi_1)$ of $-\mathcal{L}$ given in \eqref{generator-l-ad}. We
apply Algorithm~\ref{algo-1} with $K=1$, where we train the neural network
using the Adam optimization method for $J=20000$ training steps.  In all these
$20000$ steps, we use the batch-size $B=10000$, the learning rate $r=0.001$, and the penalty constant $\alpha=20$.
The mean of the first eigenvalue estimated in the last $4000$ training steps is 
\begin{equation}
  \lambda_1 = \SI{0.047}{\nano\s^{-1}}\,, 
  \label{ex2-eig-1st}
\end{equation}
with the sample standard deviation $\SI{0.005}{\nano\s^{-1}}$.  The eigenfunction $\varphi_1$ approximated by the trained neural network at the end of the training procedure is shown in Figure~\ref{fig-ad-eigenfunction-1st}.
Specifically, in the left (right) plot in Figure~\ref{fig-ad-eigenfunction-1st},
representative states in the training (test) data are placed in the angle space according to their dihedral angles $\phi_1, \phi_2$
and are colored according to the values of the first eigenfunction
$\varphi_1$. One clearly observes that the first eigenfunction $\varphi_1$
given by Algorithm~\ref{algo-1} is close to a constant within each of the
metastable regions (see Figure~\ref{fig-ad-free-energy}). The profile of
$\varphi_1$ separates the conformation C7ax from the other two conformations C5 and C7eq.
Moreover, comparing the two plots in Figure~\ref{fig-ad-eigenfunction-1st}, we
see that the eigenfunction $\varphi_1$ has very similar profiles on both
the training data and the test data, implying that the trained neural network (therefore the eigenfunction) has satisfactory generalizability.
To further verify the numerical estimation of the eigenvalue $\lambda_1$ in \eqref{ex2-eig-1st}, 
we have repeated the numerical study with a different set of training data,
sampled under the mean force that is rescaled by $\eta = 0.8$ (correspondingly, $V_{\mathrm{bias}} = 0.8 V_{\mathrm{PMF}}$).
In this case, the mean of $\lambda_1$ estimated in the last $4000$ training steps is
$\lambda_1 = \SI{0.044}{\nano\s^{-1}}$, with the sample standard deviation
$\SI{0.004}{\nano\s^{-1}}$. Moreover, a numerical study was carried out using a
larger network architecture $\mathcal{N}=(30, 25, 25, 25, 25, 1)$, which
yields the mean value $\lambda_1 = \SI{0.045}{\nano\s^{-1}}$, with the sample
standard deviation $\SI{0.006}{\nano\s^{-1}}$. These numerical experiments confirm that the numerical estimation in \eqref{ex2-eig-1st} is stable
with different choices of training data and neural network architectures.
\begin{figure}[t!]
  \centering
  \includegraphics[width=0.95\textwidth]{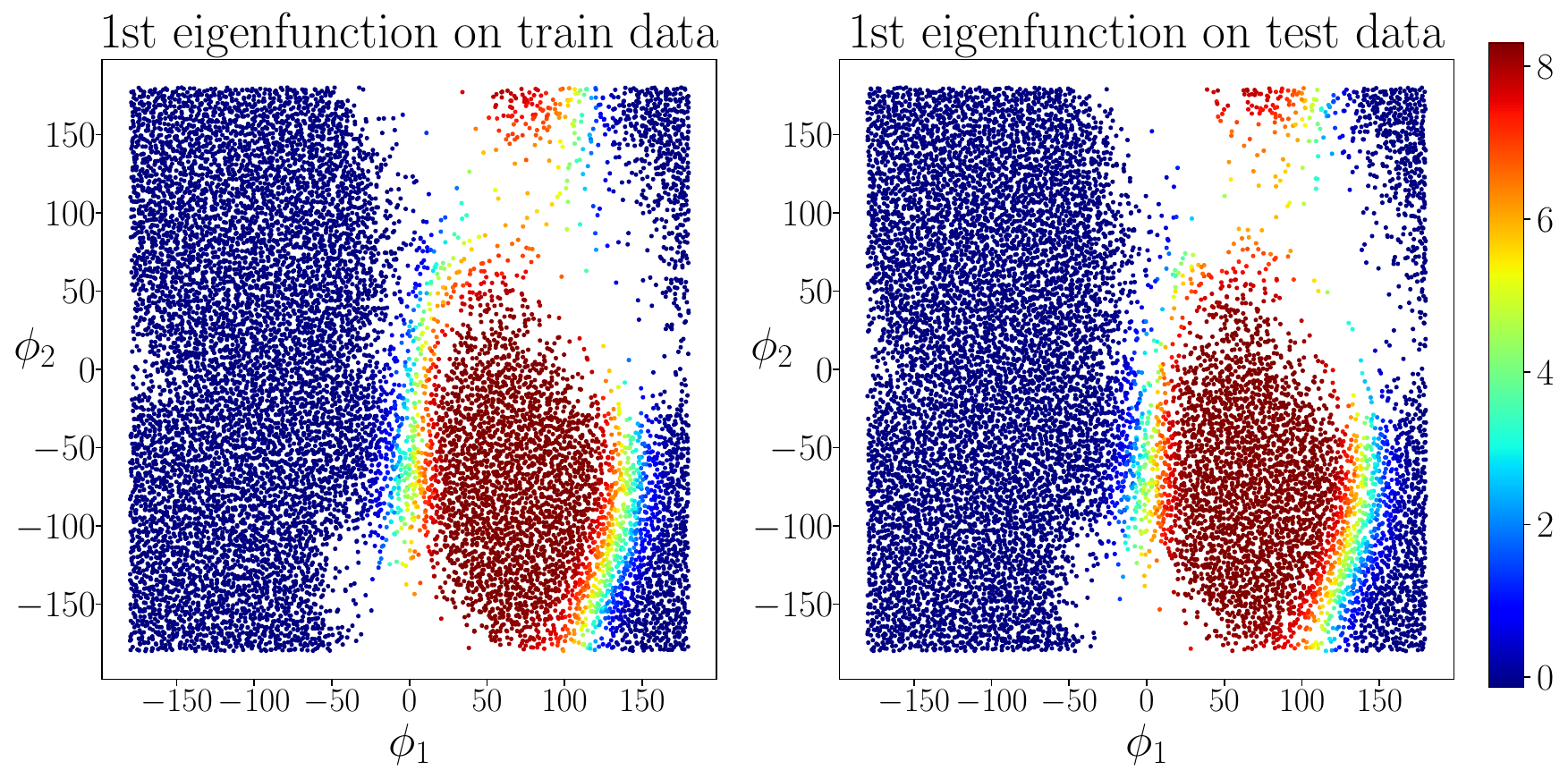}
  \caption{
For the alanine dipeptide example, the first eigenfunction $\varphi_1$ learned by training
  the neural network using Algorithm~\ref{algo-1} is evaluated on the training data
  (left) and on the test data (right). Selected states in the data sets are positioned as points according to
  their dihedral angles $\phi_1, \phi_2$, and are colored according to the values of the first eigenfunction $\varphi_1$.
  \label{fig-ad-eigenfunction-1st}}
\end{figure}

We have also computed the second eigenpair $(\lambda_2, \varphi_2)$ by applying Algorithm~\ref{algo-1} with $K=2$. 
Knowing a priori that the magnitude of the second eigenvalue $\lambda_2$
(which corresponds to the transition between C5 and C7eq; see the discussion below) is much
larger than that of $\lambda_1$ in \eqref{ex2-eig-1st}, in this test we choose the coefficients $\omega_1=1.0$ and $\omega_2=0.05$. All the other parameters are the same as those used in the previous test for computing the first eigenpair.  
After training the neural networks, we obtain numerical results of the first two eigenpairs. 
For the first eigenpair, both the estimation of $\lambda_1$ and the profile of
the eigenfunction $\varphi_1$ are very close to the results obtained in the previous
test. See \eqref{ex2-eig-1st} and Figure~\ref{fig-ad-eigenfunction-1st}, respectively.
For the second eigenpair, the mean of the eigenvalue $\lambda_2$ estimated in the last $4000$ training steps is 
\begin{equation}
  \lambda_2 = \SI{23.92}{\nano\s^{-1}}\,,
  \label{ex2-eig-2nd}
\end{equation}
with the sample standard deviation $\SI{0.60}{\nano\s^{-1}}$. Similar as in
the previous test, the left (right) plot of
Figure~\ref{fig-ad-eigenfunction-2nd} shows the second eigenfunction
$\varphi_2$ at representative states in the training (test) data set.  In
contrast to the first eigenfunction $\varphi_1$ (Figure~\ref{fig-ad-eigenfunction-1st}), 
the values of the second eigenfunction $\varphi_2$ have different signs in the
two conformational regions corresponding to C5 and C7eq (see Figure~\ref{fig-ad-free-energy} and
recall the periodic boundary conditions). This
indeed confirms that the second eigenpair is linked to the transitional events
of alanine dipeptide between the two conformations C5 and C7eq.
The fact that the estimated second eigenvalue $\lambda_2$ in \eqref{ex2-eig-2nd} is much larger than $\lambda_1$ in \eqref{ex2-eig-1st} is also consistent with the fact that the transition between C5 and C7eq is much more frequent than the transition to C7ax.

\begin{figure}[t!]
  \centering
  \includegraphics[width=0.95\textwidth]{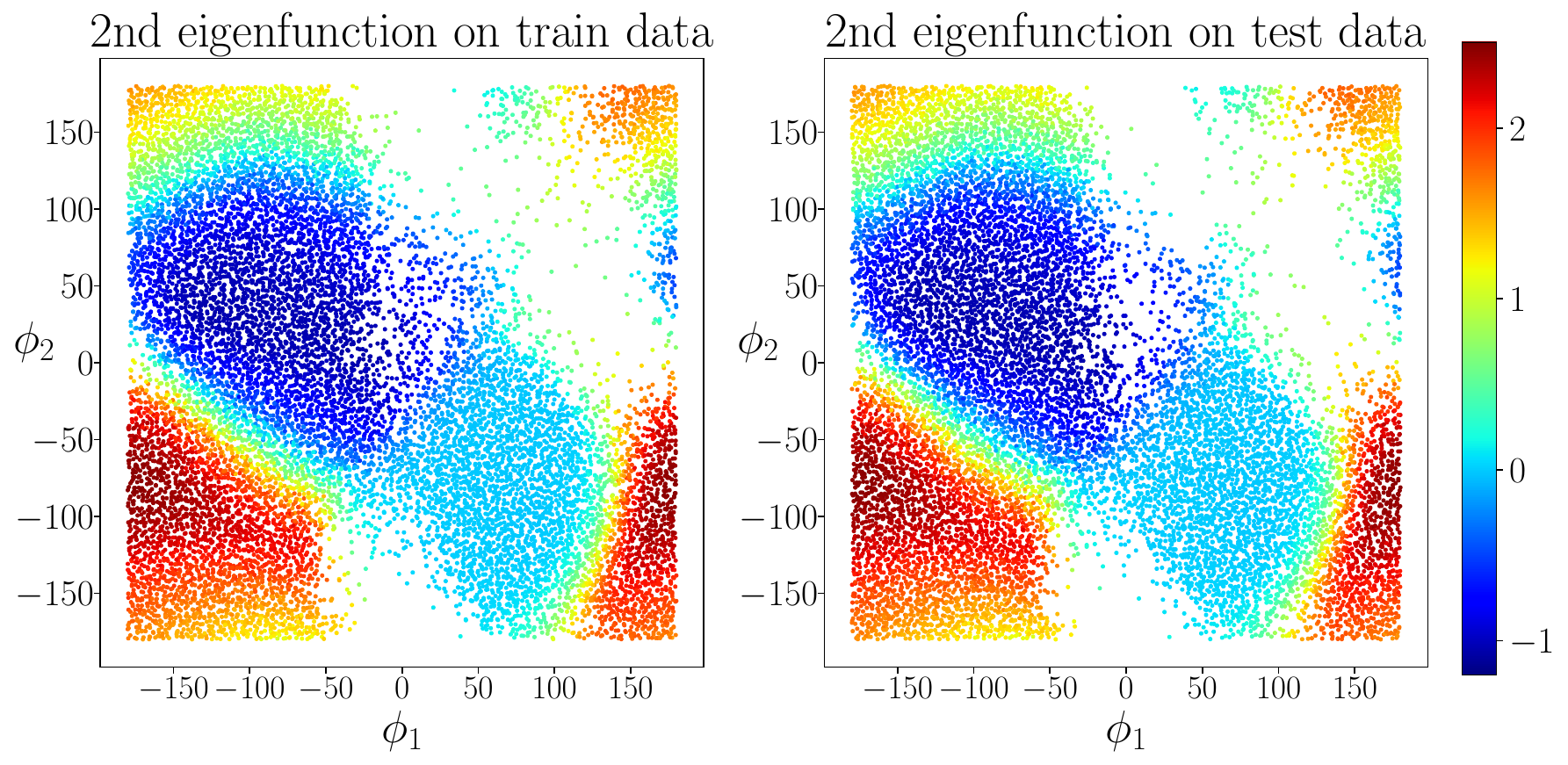}
  \caption{For the alanine dipeptide example, the second eigenfunction $\varphi_2$ learned by training neural networks
  using Algorithm~\ref{algo-1} is evaluated on the training data (left) and on the test data (right). 
  Selected states in the data sets are positioned as points according to their
  dihedral angles $\phi_1, \phi_2$, and are colored according to the values of the second eigenfunction $\varphi_2$.
  \label{fig-ad-eigenfunction-2nd}}
\end{figure}

\section{Conclusion and Discussions}
\label{sec-discuss}
In this paper, we have studied numerical methods for solving high-dimensional eigenvalue PDEs by training artificial neural networks. The
algorithm proposed in this work allows us to compute multiple eigenvalues and
the corresponding eigenfunctions of the eigenvalue PDEs. The numerical
experiments on two concrete examples demonstrate the capability of the method
in estimating large timescales and unveiling the transition mechanism of high-dimensional metastable diffusion processes.
In the following we discuss several related research topics that we would like to consider in future work. 

\textit{Learning collective variables for complex molecular systems.}
Various approaches have been developed in recent years for finding good
collective variables of molecular systems~\cite{vampnet,machine-learning-force-field-cg-variables,chen-tan-ferguson-cv-discovery-2018,chen-ferguson-on-the-fly-cv-discovery-2018,chasing-cv-autoencodes-biased-traj}.  Related to the current work, the previous
work~\cite{effective_dyn_2017} argued that the leading
eigenfunctions of the system's generator provide collective variables that are optimal in certain sense.  In future work, we will 
apply the algorithm developed in the current work in devising novel numerical
approaches for the discovery of the collective variables of molecular systems.

\textit{Application to eigenvalue problems in quantum physics.} 
Schr{\"o}dinger eigenvalue problems play a central role in quantum physics.
Thanks to the powerful capability of neural networks, numerical algorithms for solving
high-dimensional Schr{\"o}dinger eigenvalue problems are gaining research
attentions in recent years~\cite{schroedinger-eigenpde,HAN2019-solve-many-electron-schroedinger-equation,dnn-electron-schroedinger-noe,ab-initio-many-electron-schrodinger-dnn,nuesken-richter,jin-mattheakis-protopapas-pinn-quantum}.
Also see~\cite{LAGARIS19971} for an earlier work. In future, we will consider applying our numerical algorithm in solving quantum eigenvalue problems.
In particular, its potential application in computing the excited states of quantum systems will be explored. 

\textit{Alternative methods for training under constraints.} In our algorithm, the
constraints on the eigenfunctions are imposed softly using penalty method. This leads
to a simple unconstrained training task, which however involves a penalty constant $\alpha$ that has to be determined empirically. 
Although we find that our algorithm performs well in practice for a wide range
of $\alpha$, it is expected that the choice of $\alpha$ indeed plays a role in the training
procedure. Very large $\alpha$ would introduce stiffness to the problem which
in turn restricts the size of the learning rate in training, whereas a very
small $\alpha$ would make the constraints ineffective.
In future work, we will study the dependence of the algorithm on $\alpha$, as
well as alternative methods for handling constraints, such as sampling schemes
with constraints~\cite{zhang2017,multiple-projection-mcmc-submanifolds}, the
conditional gradient method and its variants~\cite{pmlr-v80-qu18a,gradient-projection-cgm-constrained-nonconvex},
and the constrained training algorithms for neural networks~\cite{leimkuhler-constraint-based-training}.

\textit{Application to more sophisticate molecular dynamics applications.} In this work we
have tested our algorithm on simple alanine dipeptide molecular system. 
Based on these experience, more sophisticate molecular systems will be studied in the next step.

\section*{Acknowledgments} 
W.\ Zhang thanks Tony Leli\`evre for fruitful discussions on the numerical
treatment of rotational and translational invariance in the alanine dipeptide
example. T.\ Li is supported by the NSFC under grant Nos. 11421101 and
11825102. The work of C.\ Sch\"utte and W.\ Zhang is supported by the DFG under Germany's Excellence Strategy-MATH+: The Berlin Mathematics Research Centre (EXC-2046/1)-project ID:390685689.
\appendix
\section{Proofs in Section~\ref{sec-theoretical-background}}
\label{sec-app-proof}
In this section, we prove Proposition~\ref{prop-spectrum} and
Theorem~\ref{thm-variational-form} in Section~\ref{sec-theoretical-background}.

  \begin{proof}[Proof of Proposition~\ref{prop-spectrum}]
For any $f \in L^2_0(\mu)$, Lemma~\ref{lemma-poisson} implies that $g=(-\mathcal{L})^{-1}f\in \mathcal{H}^1$.
Using Cauchy-Schwarz inequality and applying \eqref{poincare-ineq} to~$g$, we find
 \begin{equation*}
\|g\|^2_1 = \langle g, -\mathcal{L}g\rangle_\mu \le
\|g\|_\mu \|\mathcal{L}g\|_\mu \le 
   \sqrt{\frac{1}{\lambda}} \langle g,
   (-\mathcal{L})g\rangle_\mu^{\frac{1}{2}}\, \|\mathcal{L}g\|_\mu =
   \sqrt{\frac{1}{\lambda}} \|g\|_1 \|\mathcal{L} g\|_\mu\,,
 \end{equation*}
 which implies that $\|g\|_1 \le \sqrt{\frac{1}{\lambda}} \|\mathcal{L}g\|_\mu$, or equivalently, 
  \begin{equation}
    \|(-\mathcal{L})^{-1}f\|_1 \le \sqrt{\frac{1}{\lambda}}\|f\|_\mu\,, \quad \forall~ f \in L^2_0(\mu)\,.
    \label{l1-bounded-by-l0}
  \end{equation}

    To show that the operator $(-\mathcal{L})^{-1}:
    L^2_0(\mu)\rightarrow L^2_0(\mu)$ is compact (see \cite[Lemma 6.9]{teschl2009mathematical} and~\cite[Section
     VI.5]{reed1981functional} for equivalent definitions of compact operators),
    we consider any sequence of functions $(f_i)_{i\ge 1}$ that are bounded in $L^2_0(\mu)$. Define $g_i=(-\mathcal{L})^{-1}f_i$ for $i \ge 1$.
    The inequality \eqref{l1-bounded-by-l0} implies that the sequence $(g_i)_{i\ge 1}$ is bounded in $\mathcal{H}^1$.
    Since the embedding $\mathcal{H}^1 \hookrightarrow L_0^2(\mu)$ is compact
    by Lemma~\ref{lemma-poisson}, there is a subsequence of $(g_i)_{i\ge 1}$ which converges in $L^2_0(\mu)$.
    This shows that $(-\mathcal{L})^{-1}$ is a compact operator. 

    Concerning the second item, note that the first item implies that
    the operator $(-\mathcal{L}-\lambda I)^{-1}$ is compact for $\lambda = 0$. 
    Applying~\cite[Theorem XIII.64]{reed1978methods}, we know that there
    exists an orthonormal basis $(\varphi_i)_{i \ge 1}$ of $L^2_0(\mu)$, such that 
    $\varphi_i \in D(\mathcal{L})$ and $-\mathcal{L}\varphi_i=\lambda_i
    \varphi_i$ for $i \ge 1$, where $\lambda_1 \le \lambda_2 \le \cdots$ and $\lim_{i\rightarrow +\infty} \lambda_i =+\infty$. 
From this fact, it is not difficult to argue that the spectrum of $-\mathcal{L}$ consists of the discrete eigenvalues $(\lambda_i)_{i\ge 1}$.
  \end{proof}

\begin{proof}[Proof of Theorem~\ref{thm-variational-form}]
  Let $f_1, f_2, \dots, f_K\in \mathcal{H}^1$ be $K$ functions such that \eqref{f-orthonormal} holds.
  Using the fact that $\Sigma$ \eqref{mat-sigma} is a diagonal matrix and the diagonal elements of $F^{(K)}(f_1,f_2,\dots,
  f_K)$ in \eqref{def-f-k} are $\mathcal{E}(f_1), \mathcal{E}(f_2),\dots
  \mathcal{E}(f_K)$ (see \eqref{energy}), we find
  \begin{equation}
 \sum_{i=1}^K \omega_i \mathcal{E}(f_i) = \mbox{\textnormal{tr}} \big(\Sigma F^{(K)}(f_1,f_2,\dots, f_K)\big) \,,
    \label{trace-formula}
  \end{equation}
  which is the second equality of \eqref{variational-all-first-k}.

  Next, we show the first identity in \eqref{variational-all-first-k}. Using 
\eqref{def-f-k} and applying the Poincar{\'e} inequality \eqref{poincare-ineq}, we find that 
  \begin{equation}
    c^T F^{(K)}(f_1, f_2, \dots, f_K)c = \mathcal{E}\Big(\sum_{i=1}^K c_i f_i\Big) \ge \lambda \Big\|\sum_{i=1}^K c_i f_i\Big\|_\mu\,,\quad \forall c \in
    \mathbb{R}^K\,,
    \label{cfc}
  \end{equation}
  for some $\lambda > 0$. Since $f_1, \dots, f_K$ are linearly independent due to \eqref{f-orthonormal},
  the inequality \eqref{cfc} implies that $F^{(K)}(f_1, f_2, \dots, f_K)$ is positive definite, 
  and we denote its eigenvalues as $0 < \widetilde{\lambda}_1 \le \widetilde{\lambda}_2 \le \dots \le \widetilde{\lambda}_K$. 
  Applying Ruhe's trace inequality~\cite[H.1.h, Section H, Chapter 9]{marshall-inequalities-book},  we obtain from \eqref{trace-formula} that
  \begin{equation}
    \sum_{i=1}^K\omega_i \mathcal{E}(f_i) = \mbox{\textnormal{tr}} \big(\Sigma F^{(K)}(f_1,f_2,\dots, f_K)\big) \ge \sum_{i=1}^K \omega_i \widetilde{\lambda}_i \,.
    \label{bound-1}
  \end{equation}
 Let us show that $\widetilde{\lambda}_k \ge \lambda_k$ for $k\in \{1,2,\dots,
  K\}$. For this purpose, applying the min-max principle for symmetric matrices gives
  \begin{equation}
    \widetilde{\lambda}_k = \min_{S_k} \max_{c\in S_k, |c|=1}
    c^TF^{(K)}(f_1,f_2,\dots, f_K) c= \min_{S_k} \max_{c\in S_k, |c|=1} \mathcal{E}\big(\sum_{i=1}^K c_i f_i\big)\,,
    \label{lambda-k-tilde}
  \end{equation}
  where $S_k$ goes over all $k$-dimensional subspaces of $\mathbb{R}^K$ and
  the second equality follows from direct calculation using \eqref{def-f-k}. 
Since $(f_i)_{1\le i\le K}\subset \mathcal{H}^1$ satisfies the orthonormality condition \eqref{f-orthonormal},
  each $k$-dimensional subspace $S_k\subset \mathbb{R}^K$ defines a $k$-dimensional subspace of $\mathcal{H}^1$ by $\widetilde{H}_k=\big\{\sum_{i=1}^K c_if_i\,|\, c\in S_k\big\}$ such that  $\widetilde{H}_k\subset \mbox{span}\{f_1,f_2, \dots, f_K\}$. 
  On the contrary, every $k$-dimensional subspace $\widetilde{H}_k \subset
  \mbox{span}\{f_1,f_2, \dots, f_K\}$ can be written in this way for some $k$-dimensional subspace $S_k\subset \mathbb{R}^K$.
 Therefore, using \eqref{lambda-u-variational-ith}, we find from \eqref{lambda-k-tilde} that 
  \begin{equation}
    \widetilde{\lambda}_k = \min_{S_k} \max_{c\in S_k, |c|=1} \mathcal{E}\big(\sum_{i=1}^k c_k f_k\big)
    = \min_{\widetilde{H}_k} \max_{f \in \widetilde{H}_k, \|f\|_\mu=1} \mathcal{E}(f) \ge \lambda_k\,.
    \label{cmp-eigen-of-f-to-real}
  \end{equation}
  Combining \eqref{cmp-eigen-of-f-to-real} and \eqref{bound-1}, gives 
  \begin{equation*}
    \sum_{i=1}^K\omega_i \mathcal{E}(f_i) = \mbox{\textnormal{tr}} \big(\Sigma
    F^{(K)}(f_1,f_2,\dots, f_K)\big) \ge \sum_{i=1}^K \omega_i \widetilde{\lambda}_i  \ge \sum_{i=1}^K \omega_i \lambda_i\,.
  \end{equation*}
    Since the eigenfunctions $(\varphi_i)_{1\le i\le K}$ satisfy \eqref{f-orthonormal}
    and we have the identity $\sum_{i=1}^K\omega_i \mathcal{E}(\varphi_i) = \sum_{i=1}^K \omega_i
    \lambda_i$, we conclude that the first equality in \eqref{variational-all-first-k} holds and the minimum is achieved when
  $f_i=\varphi_i$ for $i\in \{1,2,\dots, K\}$.
\end{proof}

\bibliographystyle{siamplain}
\bibliography{reference}
\end{document}